\documentclass[11pt,reqno]{amsproc}

\title[]{Electric inertia and ideal magnetic reconnection in 2D}

\author{Peter Constantin}
\address{Department of Mathematics, Princeton University, Princeton, NJ 08544}
\email{const@math.princeton.edu}
\author{Zhongtian Hu}
\address{Department of Mathematics, Princeton University, Princeton, NJ 08544}
\email{zh1077@princeton.edu}
\usepackage[margin=1in]{geometry}
\usepackage{amsmath, amsthm, amssymb}
\usepackage{times}
\usepackage{color}
\usepackage{hyperref,comment}

\newcommand{\pa}{\partial}
\newcommand{\la}{\label}
\newcommand{\fr}{\frac}
\newcommand{\na}{\nabla}
\newcommand{\be}{\begin{equation}}
\newcommand{\ee}{\end{equation}}
\newcommand{\ba}{\begin{array}{l}}
\newcommand{\ea}{\end{array}}
\newcommand{\Rr}{{\mathbb R}}
\newtheorem{thm}{Theorem}
\newtheorem{prop}{Proposition}
\newtheorem{lemma}{Lemma}
\newtheorem{rem}{Remark}
\newtheorem{defi}{Definition}
\newtheorem{cor}{Corollary}
\newcommand{\beg}{\begin}
\newcommand{\ov}{\overline}

\newcommand{\eps}{\epsilon}
\newcommand{\R}{\mathbb{R}}
\newcommand{\B}{\mathbb{B}}
\newcommand{\U}{\mathbb{U}}

\newcommand{\D}{\Delta}
\newcommand{\V}{\mathbb{V}}

\newcommand{\ze}{\zeta}
\newcommand{\p}{\partial}
\newcommand{\pt}{\p_t}
\newcommand{\palpha}{\p_\alpha}
\newcommand{\pb}{\p_\beta}
\newcommand{\calH}{\mathcal{H}}
\newcommand{\calD}{\mathcal{D}}
\newcommand{\calK}{\mathcal{K}}

\newcommand{\tsig}{\tilde{\sigma}}

\DeclareMathOperator{\supp}{supp}
\DeclareMathOperator{\Int}{Int}
\date{today}
\begin{document}
\begin{abstract}
We consider inertial magneto-hydrodynamic systems in 2D. We show global existence and uniqueness of smooth solutions and global existence and uniqueness of weak solutions in Yudovich class. We  prove magnetic reconnection without magnetic resistivity, for smooth solutions
and for patch solutions. This is obtained by proving merger  in corresponding systems of coupled active scalars.
\end{abstract}
\noindent
\maketitle
\section{Introduction}
Ideal magnetic fields do not change topology of magnetic lines, as a consequence of the magnetic induction equation. But change of topology of magnetic lines does occur in nature. There is a well established correlation between fast magnetic reconnection and solar flares.  A natural suggestion is that resistive MHD evolution  might be responsible for topology change, leading to magnetic reconnection events ~\cite{biskamp, taylor1986}. The reconnection of vortex lines is followed by massive energy release. 
There is an analogous problem in incompressible Navier-Stokes equations. In \cite{geoman, siam} it was argued on the basis of physical and analytical arguments that vortex reconnection is a regularizing mechanism in Navier-Stokes equations.

A different type of magnetic reconnection mechanism was proposed, based on a modification of ideal MHD due to magnetohydrodynamic inertia. This was discussed in \cite{twoplasmas} and \cite{morrison} in 2D. 
and more recently in \cite{boozer} for 3D.  The main idea is that the electronic inertia modifies the transport.

In this paper we prove that magnetic reconnection does occur without magnetic resistivity.  We take with \cite{twoplasmas} a model of electronic inertia in 2D plasma dynamics.
The magnetic field is $B= \na^{\perp}\phi$ where $\phi$ is a scalar time dependent function of two variables and $\na^{\perp} = (-\pa_2,\pa_1)$. The current density is $J= -\Delta \phi$ and a variable $F = \phi + J$ plays the role of a canonical momentum. The electron flow velocity is given by $\na^{\perp}\psi$ where
$\psi$ is a stream function. We denote by
$\omega$ the vorticity, $\omega = \Delta \psi$. The system \cite{twoplasmas} is
\be
\left\{
\ba
\pa_t F + \{\psi, F\} + \{\phi, \omega\} = 0,\\
\pa_t \omega + \{\psi, \omega\} + \{\phi, F\} =0.
\ea
\right .
\la{origeq}
\ee
Here $\{f, g\} = \pa_1f\pa_2 g-\pa_2f\pa_1g$ is Poisson bracket. 

As noted in \cite{twoplasmas}, the quantities $F+\omega$ and $F-\omega$ are each advected by incompressible velocities
\be
\pa_t (F+\omega) + \{\psi +\phi, F+\omega\} =0
\la{sigmapluseq}
\ee
and
\be
\pa_t (F-\omega) + \{\psi-\phi, F-\omega\} = 0.
\la{sigmaminuseq}
\ee
 If initial data for $F\pm\omega$ are integrable and bounded in space, that is, if both initial data are in the Yudovich class $\mathbb Y = L^1\cap L^{\infty}$, then  the solutions remain in $\mathbb Y$. Formally,
 from \eqref{sigmapluseq} and \eqref{sigmaminuseq} it follows that
 \[
 \|(F\pm\omega)(t)\|_{L^{p}} = \|(F\pm\omega)(0)\|_{L^{p}}.
 \]
 for every $1\le p\le\infty$. This is powerful information.  
 The equations are  a coupled system of active scalars \cite{geoman, siam}, $\sigma_{\pm} = F\pm\omega$ 
 advected by incompressible velocities $v_{\pm}$,
\be
\pa_t \sigma_{\pm} + v_{\pm}\cdot\na\sigma_{\pm} = 0.
\la{active}
\ee 
The velocities $v_{\pm}$ are obtained from $\sigma_{\pm}$ by a linear, time independent equation of state
 \be
 \left\{
 \ba
 v_+ =  \na^{\perp}\left ( \mathbb K \sigma_+ + \mathbb L \sigma_-\right)\\
 v_- = \na^{\perp}\left (-\mathbb L \sigma_+ - \mathbb K \sigma_-\right)
 \ea
 \right.
 \la{pmright}
 \ee
 where 
 \be
 \mathbb K  = \fr{1}{2}\left ( (-\Delta + \mathbb I)^{-1} - (-\Delta)^{-1}\right)
 \la{mathbbK}
 \ee
 and
 \be
 \mathbb L   = \fr{1}{2}\left ( (-\Delta +\mathbb I)^{-1} + (-\Delta)^{-1}\right )
 \la{mathbbL}
 \ee 
 are functions of $-\Delta$. In this paper we call this system the ``right-handed" system. We discuss also an analogous system we call ``left-handed", where $(-\Delta + \mathbb I)^{-1} $ above is replaced by $(\Delta - \mathbb I)^{-1}$, corresponding to having $F = \Delta\phi -\phi$ instead of $F = -\Delta \phi + \phi$. 
 (We choose to keep $\omega = \Delta\psi$ throughout, because of the analogy to fluid equations,  but switching to $\omega = -\Delta\psi$ and keeping $F = -\Delta \phi + \phi$ results in the same left-handed system, mutatis mutandis. ) 
The left-handed system is thus \eqref{active} together with
\be
 \left\{
 \ba
 v_+ =  -\na^{\perp}\left ( \mathbb L \sigma_+ + \mathbb K \sigma_-\right)\\
 v_- = \na^{\perp}\left (\mathbb K \sigma_+  +  \mathbb L \sigma_-\right)
 \ea
 \right.
 \la{pmleft}
 \ee
which means that  $\mathbb K$ and $\mathbb L$ have been  switched to $-\mathbb L$ and $-\mathbb K$ compared to \eqref{pmright}.

Both systems have unique global smooth solutions in $C^{1,\alpha}$ and unique global weak solutions in $\mathbb Y$. Both systems can be derived as critical points of a Lagrangian action \cite{morrison} (see also Section \ref{deriv}), both have infinitely many conserved quantities (Casimirs) and both systems are capable of producing  topology change of the magnetic field $F$ in both smooth regimes and in vortex patch regimes. The two systems are similar, but not at all the same. In the original right-handed system described above, the diagonal terms $\mathbb K$ and $-{\mathbb K}$ would vanish if $\mathbb I$ would be replaced by zero. In fact, the operator $\mathbb K$ is smoothing of degree four (while $\mathbb L$ is smoothing of degree two) and  is small at small distances. Removing the screening term $\mathbb I$, we have at leading order  (at  high frequencies, or small scales) 
\be
\left\{
\ba
\pa_t\sigma_{+}  - (\na^{\perp}\Delta^{-1}\sigma_-)\cdot\na \sigma_+ = 0,\\
\pa_t\sigma_{-} + (\na^{\perp}\Delta^{-1}\sigma_+)\cdot\na\sigma_- = 0,
\ea
\right.
\la{right}
\ee
a system in which one scalar drives the other.  For the left-handed system, removing the screening term
$\mathbb I$ results in the system
\be
\left\{
\ba
\pa_t\sigma_{+}  +(\na^{\perp}\Delta^{-1}\sigma_+)\cdot\na \sigma_+ = 0,\\
\pa_t\sigma_{-}  - (\na^{\perp}\Delta^{-1}\sigma_-)\cdot\na\sigma_- = 0,
\ea
\right.
\la{left}
\ee
a system in which the scalars obey incompressible Euler equations and do not interact. Here the origin of the
topology change is much simpler and clearer, the merger of the scalars is due to their lack of interaction.

Both the left-handed and the right-handed unscreened systems have the full Euler two-parameter scaling properties, meaning that if $\sigma_{\pm}(x,t)$ is a solution, then so is $\fr{1}{T}\sigma_{\pm}\left(\fr{x}{L}, \fr{t}{T}\right)$ for any positive $T$ and any positive $L$. By contrast, the screened systems have a distinguished natural length scale (related to the electronic skin depth, and taken to be the unit in this paper). In order to prove our change of topology results, we first understand them in unscreened leading order systems, and then prove they hold in the screened systems as well, by taking advantage of the smallness of the difference between unscreened and screened kernels at small distances, and by perturbing the merger of small scale solutions of the unscreened systems.

For the left-handed system where the unscreened solutions are non-interacting 2D Euler solutions we prepare two patches, a rotating Kirkhhoff ellipse and a disk, which are initially disjoint and which overlap later during the evolution. We then show that suitably scaled down, these two patches provide initial data for finite time merger of solutions of the screened equations.   

For the right-handed system we need to establish first the unscreened merger. This is done by choosing a
specific configuration with symmetry. Both $\sigma_+$ and $\sigma_-$ are taken odd in $x_1$ and we also assume that $\sigma_-$ is obtained from $\sigma_+$ by reflection across the $x_1$ axis. These symmetries are preserved by the right-handed equations (both screened and unscreened). Then we consider initial data for $\sigma_+$ that are compactly supported in the upper half plane and have a single sign in the upper right quadrant. Thus, the supports of $\sigma_+$ and $\sigma_-$ are initially disjoint. We obtain the finite time merger by analyzing the evolution of local barycenters. We show that the support of $\sigma_+$ rapidly descends and crosses the $x_1$ axis, while by symmetry, $\sigma_-$ ascends across it, and merger occurs. We show then that this behavior is produced also in the screened right-handed equations by analyzing the same objects. We also provide a more general merger stability analysis (see Section \ref{mergstabi}) where we show that merger in unscreened equations from smooth disjoint compactly supported initial data implies merger of small scale solutions in the screened equations. We consider further the stability of merger under resistive approximations, and show that, in the smooth case, solutions of resistive equations converge strongly to solutions of the ideal equations, and thus the ideal merger is stable to small resistive perturbations. This result is expected to hold also in Yudovich class, on a basis of an inviscid limit in the spirit of \cite{cde}, a topic which will be pursued in later work. 

The organization of the paper is as follows. In Section \ref{globsmooth} we prove global existence and uniqueness of smooth solutions. Global existence and uniqueness of Yudovich class solutions is presented in Section \ref{yud}. Merger of solutions of the screened left-handed equations for vortex patch initial data is presented in Section \ref{patchmerge}. The construction of merger of smooth solutions of unscreened and of screened right-handed equations is in Section \ref{recright}. Stability of merger is discussed in Section \ref{mergstabi}. The limit of zero resistivity for smooth solutions is proved in Section \ref{invlim}. A unified formal derivation of both right and left handed equations is presented in Section \ref{deriv}.

 \section{Global existence and uniqueness}\la{globsmooth}
 The results in this section are valid for both the right-handed and left-handed models, and their proofs are identical.  We present the right-handed case. We consider the operators
 \be
 \mathbb U (\omega) = \na^{\perp}(\D^{-1})\omega
 \la{mathbbU}
 \ee
 and
 \be
 \mathbb B( F) = \na^{\perp}(\mathbb I - \D)^{-1}F.
 \la{mathbbB}
 \ee
 In Eulerian frame, setting
 \be
 \sigma_{\pm} = F\pm \omega
 \la{sigmapm}
 \ee
 and
 \be
 \sigma = (\sigma_+, \sigma_-),
 \la{sigma}
 \ee
 we have
 \be
 \pa_t \sigma_{\pm} + \mathbb V_{\pm}(\sigma)\cdot\na \sigma_{\pm} =0,
 \la{sigmaeq}
 \ee
 where 
 \be 
 \mathbb V_{\pm}(\sigma) = \fr{1}{2}\left (\mathbb U (\sigma_{+}-\sigma_{-}) \pm \mathbb B(\sigma_+ +\sigma_-)\right )
 \la{Vpm}
 \ee
 which holds in view of the fact that
 \be
 \omega = \fr{1}{2}(\sigma_{+}-\sigma_{-}), \quad F=  \fr{1}{2}(\sigma_{+}+\sigma_{-}),
 \la{omegaFsigmas}
 \ee
 and so
 \be
 \mathbb V_{\pm}(\sigma) = \mathbb U(\omega)\pm \mathbb B(F),
 \la{mathbbV}
 \ee
and \eqref{sigmaeq} is just the system \eqref{sigmapluseq}, \eqref{sigmaminuseq}.

 We consider a pair
 \be
 X= (X_+, X_-)
 \la{X}
 \ee
 of Lagrangian maps
 \be
 X_{\pm}(\cdot, t):\Rr^2\to\Rr^2
 \la{xpm}
 \ee
 and fields  $\tau = (\tau_+, \tau _-) :\Rr^2 \to \Rr^2$ which are time independent and correspond to the initial data,
 \be
 \tau (a) = (F_0(a) + \omega_0(a), F_0(a) - \omega_0(a) ) = \sigma_0(a)
 \la{taupm}
 \ee
 We solve
 \be
 \pa_t X_{\pm} = \mathcal U_{\pm}(X)
 \la{Xeq}
 \ee
 where 
 \be
 \mathcal U_{\pm}(X_{\pm}) = \mathbb V_{\pm} (\tau_+\circ X_{+}^{-1}, \tau_ -\circ X_-^{-1})\circ X_{\pm}
 \la{mathcalU}
 \ee
 with initial data
 \be
 X_{\pm}(a, 0) = a.
 \la{xid}
 \ee
 Using the method of \cite{hydro}
 we have
 \beg{thm} \label{thm:Calphalwp}
 Let $1<p<\infty$ and $0<\alpha<1$, and let $\tau_{\pm}\in (C^{1+\alpha}\cap L^p)(\Rr^2)$. There exists $T>0$ depending only on the norms $\|\tau_{\pm}\|_{L^p\cap C^{1+\alpha}}$ and 
 $\lambda_{\pm}\in Lip(0,T; C^{1+\alpha}(\Rr^2))$, with $\lambda_{\pm}(a,0) = 0$,  such that $X_{\pm}(a,t) = a + \lambda_{\pm}(a,t)$ solve
 \eqref{Xeq}.  
 
 The solution $X(a,t) = (X_+(a,t), X_-(a,t))$ depends in Lipschitz manner on initial data in the norm of $||| X ||| = \sup_{t\in[0,T]}\|\lambda_{\pm}(t) \|_{C^{1+\alpha}} $, i.e. if $\widetilde X$ solves the Lagrangian system with initial data $\widetilde\tau$ then
 \be
 ||| X- \widetilde X||| \le C\|\tau-\widetilde\tau\|_{C^{1+ \alpha}\cap L^p}.
 \ee
 \end{thm}
 This theorem implies that the solution of \eqref{sigmaeq} is given by
 \be
 \sigma_{\pm}(x,t) = \tau_{\pm}(X_{\pm}^{-1}(x,t))
 \la{sigmatau}
 \ee
 for $t\in [0,T]$. We note that 
 \be
 \|\sigma_{\pm}(t)\|_{L^p} = \|\tau_{\pm}\|_{L^p}
 \la{normp}
 \ee
 holds for $0\le t\le T$ and for $1\le p\le\infty$ because $X_{\pm}$ are volume preserving.
 
 \beg{rem} In contrast with the single fluid Lagrangian evolution, where the Lipschitz dependence is only
 on $\|\tau-\widetilde \tau\|_{C^{\alpha}\cap L^p}$, for two-fluid equations there is a loss of a derivative.  
 \end{rem}
 
  The global existence of solutions now follows from a priori estimates and uniqueness, in Eulerian coordinates. We use the inequality
  \be
  \|\na \mathbb V_{\pm}(\sigma)\|_{L^{\infty}} \le C\|\sigma\|_{L^{\infty}} \log\left (2 + \fr{\|\na\sigma\|_{L^{\infty}}\|\sigma\|_{L^1}^{\fr{1}{2}}}{\|\sigma\|_{L^{\infty}}^{\fr{3}{2}}}\right).
  \la{extra}
  \ee
We obtain
\beg{thm} \label{thm:Calphagwp}
Let $\tau_{\pm}\in (L^1\cap C^{1+\alpha})(\Rr^2)$. There exists $C>0$ and a unique solution $\sigma(x,t) = (\sigma_+(x,t), \sigma_-(x, t))$ of \eqref{sigmaeq} with initial data $\sigma_{\pm}(x, 0) = \tau_{\pm}(x)$, satisfying
\be
\|\sigma(t)\|_{C^{1+\alpha}\cap L^1} \le \|\tau\|_{C^{1+\alpha}\cap L^1}\exp\left[ C\left( e^{Ct\|\tau\|_{L^{\infty}}} -1\right )\log\left (2 + \fr{\|\na\tau\|_{L^{\infty}}\|\tau\|_{L^1}^{\fr{1}{2}}}{\|\tau\|_{L^{\infty}}^{\fr{3}{2}}}\right)\right]
\la{aprb}
\ee
for all $t\ge 0$.
\end{thm}
\beg{proof}
We indicate only the salient elements of the proof and provide useful a priori bounds. On the basis of \eqref{sigmaeq} we have the a priori information that
\be
\|\sigma_{\pm}(t)\|_{L^p} = \|\tau_{\pm}\|_{L^p}
\la{sigmalp}
\ee
holds for all $1\le p\le\infty $. Then using the stretching equation
\be
(\pa_t + \mathbb V_{\pm}(\sigma) \cdot\na)(\na^{\perp}\sigma_{\pm}) = (\na \mathbb V_{\pm}(\sigma) )(\na^{\perp}\sigma_{\pm})
\la{nasigmaeq}
\ee
and the bound \eqref{extra}, we deduce using the logarithmically superlinear ODE that
\be
2 + \|\tau\|_{L^1}^{\fr{1}{2}}\|\tau\|_{L^{\infty}}^{-\fr{3}{2}}\|\na\sigma_{\pm}(t)\|_{L^{\infty}} \le \exp\left [e^{Ct\|\tau\|_{L^{\infty}}} \log\left (2 + \fr{\|\na\tau\|_{L^{\infty}}\|\tau\|_{L^1}^{\fr{1}{2}}}{\|\tau\|_{L^{\infty}}^{\fr{3}{2}}}\right)\right]
\la{nasigmabound}
\ee
holds a priori.  Returning to \eqref{extra} we have
\be
\|\na \mathbb V_{\pm}(\sigma)\| \le C\|\tau\|_{L^{\infty}}\left [e^{Ct\|\tau\|_{L^{\infty}}} \log\left (2 + \fr{\|\na\tau\|_{L^{\infty}}\|\tau\|_{L^1}^{\fr{1}{2}}}{\|\tau\|_{L^{\infty}}^{\fr{3}{2}}}\right)\right]
\la{naVlifty}
\ee
Because of the equation
\be
\fr{d}{dt}(\na X_{\pm}) = (\na \mathbb V_{\pm}(\sigma)\circ X_{\pm})\na X_{\pm}
\la{naXeq}
\ee
we know that the Lagrangian maps $X_{\pm}$ are globally defined, invertible and they and their inverses are Lipschitz, with a priori bounds of their Lipschitz constants that are uniform in space and doubly exponential in time.
\be
\max\{\|\na X_{\pm}(t)\|_{L^{\infty}}\; ,  \|\na X^{-1}(t)\|_{L^{\infty}}\} \le A(t)
\la{naXbound}
\ee
where the doubly exponential amplification factor $A(t)$ is 
\be
A(t) = \exp\left[ C\left( e^{Ct\|\tau\|_{L^{\infty}}} -1\right )\log\left (2 + \fr{\|\na\tau\|_{L^{\infty}}\|\tau\|_{L^1}^{\fr{1}{2}}}{\|\tau\|_{L^{\infty}}^{\fr{3}{2}}}\right)\right]
\la{At}
\ee
as it follows from  \eqref{naVlifty} and \eqref{naXeq} integrating in time. Note that  $A(t)$ is a priori controlled 
in time (by a double exponential) and depends only on  $\|\tau\|_{W^{1, \infty}}$. Note also that the expression is nondimensional.

Returning to \eqref{nasigmaeq} and using \eqref{naVlifty} we have also the bounds
\be
\|\na\sigma_{\pm}(t)\|_{L^p} \le A(t)\|\na\tau_{\pm}\|_{L^p} 
\la{nasigmalpbA}
\ee
for all $1\le p\le\infty$.

At this point we have the information needed to pass to $C^{1+\alpha}$ norms. We use the fact composing with $X_{\pm}(t)$ or $X_{\pm}(t)^{-1}$ does not change $L^{\infty}$ norms and amplifies Holder seminorms by $A(t)$. (We replace $A(t)^{\alpha}$ by $A(t)$  here using that $A(t)\ge 1$.) Repeated compositions do not change the amplification significantly because of the doubly exponential nature of $A(t)$, so that $A(t)^2$ or $A(t)^3$ are still doubly exponential factors which are a priori controlled in time by $\|\tau\|_{W^{1,\infty}}$. 

The fact that $\sigma_{\pm}(t) = \tau_{\pm}(X_{\pm}^{-1}(t))$ is used now together with the chain rule,
\be
\na\sigma_{\pm}  = (\na \tau_{\pm}(X_{\pm}^{-1}))(\na X_{\pm}^{-1}).
\la{chain}
\ee
Now $C^{\alpha}$ is a Banach algebra with  $\|fg\|_{C^{\alpha}}\le \|f\|_{L^{\infty}}\|g\|_{C^{\alpha}} + \|g\|_{L^{\infty}}\|f\|_{C^{\alpha}}$.
We estimate
\be
\| \na \tau_{\pm}\circ X_{\pm}^{-1}\|_{C^{\alpha}} \le A(t)\|\tau\|_{C^{1+\alpha}}
\la{nasigalpha}
\ee
In order to estimate $\|\na X_{\pm}^{-1}\|_{C^{\alpha}}$ we use the fact that the inverse map solves the ODE in Lagrangian coordinates 
\be
\fr{d}{dt}(\na X_{\pm}^{-1}\circ X_{\pm})  +  (\na X_{\pm}^{-1}\circ X_{\pm})(\na \mathbb V_{\pm}(\sigma)\circ X_{\pm}) = 0.
\la{odenaA}
\ee
We use the bound \eqref{naVlifty} and also the Eulerian fact that $\|\na\mathbb V_{\pm}(\sigma)\|_{C^{\alpha}} \le C\|\sigma\|_{C^{\alpha}}$. This is due to the boundedness of Riesz and Bessel transforms in $C^{\alpha}$. Then
\be
\| \na \mathbb V_{\pm}(\sigma)\circ X_{\pm}\|_{C^{\alpha}} \le A(t)\|\tau\|_{C^{\alpha}}
\la{naVcalpha}
\ee
and consequently after judicious use of \eqref{naVlifty} and \eqref{naVcalpha} we have
\be
\|\na X_{\pm}^{-1}\|_{C^{\alpha}} \le CA(t)^3
\ee
We  omit further details. 
\end{proof}

\section{Global existence and uniqueness of  weak solutions in Yudovich class}\la{yud}

The global existence and uniqueness of weak solutions with $\sigma_\pm \in \mathbb Y = (L^1\cap L^{\infty})(\Rr^2)$ is established here. The result holds for both the right-handed and the left-handed models with similar proofs. We show here the original right-handed case. We introduce the notion of weak solution in Yudovich class. 
\begin{defi}\la{defn:weaknondeg}
    Given initial condition $\tau = (\tau_+, \tau_-) \in \mathbb Y^2$ and $T > 0$, we say that  $\sigma = (\sigma_+, \sigma_-)$  is a Yudovich class solution of \eqref{sigmaeq} on $[0,T]$ with initial data $\tau$, if 
   $\sigma_\pm \in L^\infty([0,T]; (L^1\cap L^\infty)(\R^2))$, the equation
   
   \begin{equation}\la{eq:continuity}
   \pa_t \sigma_{\pm} + \na\cdot(\V_{\pm}(\sigma)\sigma_{\pm} ) = 0
   \end{equation}
   holds in $C([0, T); W^{-1, p}(\Rr^2))$ for $1<p<\infty$, with $\V_{\pm}(\sigma)$ defined in \eqref{Vpm},
   and if 
   \[
   \tau (x) = \lim_{t \to 0}\sigma(x,t) \: \text{ in}\; (W^{-1,p}(\Rr^2))^2.
   \]
   \end{defi}
   Here, we recall \eqref{Vpm}:
     \[
     \V_{\pm}(\sigma)= K_\U * \left(\frac{1}{2}(\sigma_+ - \sigma_-)\right) \pm K_\B* \left(\frac{1}{2}(\sigma_+ + \sigma_-)\right),
     \] 
     where $K_\U, K_\B$ are kernels of the operators $\U, \B$ (cf. \eqref{mathbbU}, \eqref{mathbbB})

With the above definition, we prove the following theorem:

\begin{thm}\label{thm:Yudovich}
    Fix any initial data $\tau \in (L^1 \cap L^\infty)(\R^2)$. There exists a unique Yudovich class  solution $\sigma = (\sigma_+, \sigma_-)$ with initial data $\tau$ on $[0,T]$ for any $T \in (0,\infty)$.
    Moreover, the velocity vector fields $\V_{\pm}$ are log-Lipschitz in space. Define the Lagrangian trajectory $X(a,t) = (X_+(a,t), X_-(a,t))$ to be the (unique) solution to ODEs:
    \begin{equation}
        \label{trajectory}
        \frac{dX_\pm (a,t)}{dt} = \V_\pm (X_\pm(a,t),t),\quad X_\pm(a,0) = a.
    \end{equation}
    Then $X(a,t)$ is H\"older continuous with respect to $a$. More precisely, we have
        \begin{subequations} \la{est:holderX}
        \be
        |X_\pm^\epsilon(\alpha_1,t) -X_\pm^\epsilon(\alpha_2,t)| \le C|\alpha_1 - \alpha_2|^{\beta(t)},
        \ee
        \be
        |(X_\pm^\epsilon)^{-1}(x_1,t) -(X_\pm^\epsilon)^{-1}(x_2,t)| \le C|x_1 - x_2|^{\beta(t)},
        \ee
        \end{subequations}
    where $\beta(t) = \exp(-C\|\tau_\pm\|_{L^1 \cap L^\infty}t)$ for some universal constant $C > 0$.
\end{thm}

\begin{proof}
We remark that the existence part of Theorem \ref{thm:Yudovich} can be adapted from the single-fluid Yudovich theory for 2D Euler equation, so we provide only the outlines of the proof. The uniqueness part is more subtle and we provide a full proof in the spirit of the original proof of \cite{yudovich}. A proof using exponential integrability and Legendre transform, as in \cite{cde} is also possible, but we do not pursue this here.

The proof for existence is presented in the following steps:
\begin{enumerate}
    \item \textbf{Introduction of $\epsilon$--mollified solution.} Given $\epsilon \in (0,1)$, we let $\phi_\epsilon$ be the standard mollifier at spatial scale $\epsilon$ and define the mollified initial data as $\tau_\pm^\epsilon := \phi_\epsilon * \tau_\pm$. In view of Theorem \ref{thm:Calphalwp} and \ref{thm:Calphagwp}, we let $\sigma^\epsilon = (\sigma_+^\epsilon, \sigma_-^\epsilon)$ be the unique, global, regular solution initiated by $\tau^\eps := (\tau_+^\epsilon,\tau_-^\epsilon)$, with corresponding velocity $\V^\epsilon = (\V^\epsilon_+, \V^\epsilon_-)$.

    \item \textbf{Bounds on mollified variables.} Similar to the case of 2D Euler equation,  the following bounds are proved: given $\epsilon \in (0,1)$, $T > 0$, and $t\in(0,T)$, we have
    \begin{enumerate}
        \item (Uniform bound on mollified vorticity)
        \be\la{est:Lpweps}
        \|\sigma^\epsilon\|_{L^p} \le \|\tau^\eps\|_{L^p} \le \|\tau\|_{L^p},\quad p \in [1,\infty].
        \ee
        \item (Uniform bound on mollified velocity) 
        \be\la{est:LinftyVeps}
        \|\V_\pm^\epsilon\|_{L^\infty} \le C\|\sigma_+^\epsilon \pm \sigma_-^\epsilon\|_{L^1\cap L^\infty} \le C\|\tau\|_{L^1\cap L^\infty},
        \ee
        \be\la{est:W1pVeps}
        \|\na \V^\eps_\pm\|_{L^p} \le C(p)\|\sigma_+^\epsilon \pm \sigma_-^\epsilon\|_{L^1\cap L^\infty} \le C(p)\|\tau\|_{L^1\cap L^\infty}.
        \ee

    \end{enumerate}
    We briefly explain how to obtain the aforementioned bounds: the bound \eqref{est:Lpweps} follows directly from \eqref{normp} together with properties of the standard mollifier. The proof of \eqref{est:LinftyVeps}--\eqref{est:W1pVeps} follows from a similar argument to that treating 2D Euler equation, except for bounds concerning the kernel $K_\B$. However, from \eqref{mathbbB} we see that $K_\B = \nabla^\perp G$, where $G$ is the Bessel potential of order 2 on $\R^2$. Then $K_\B(x,y) \sim |x-y|^{-1}$ near the diagonal set and decays exponentially at spatial infinity. Therefore, we obtain similar potential-theoretical bounds for contributions by $K_\B$ to those arising from $K_\U$. 
    
    \item \textbf{Passing to the Limit.} From \eqref{est:W1pVeps} and Rellich's compact embedding theorem, we have $\V^\eps_\pm \to \V_\pm$ in $L^\infty([0,T]; L^q_{loc})$ for arbitrary $q \in (1,\infty)$ and some $\V_\pm$. Next, the uniform bound \eqref{est:Lpweps} implies that $\sigma^\eps \overset{*}{\rightharpoonup} \sigma$ in $L^\infty([0,T]; L^1\cap L^\infty)$.  
    The above convergence properties are sufficient to verify that the Biot-Savart law \eqref{Vpm} holds after passing to the limit: the operator $\sigma \mapsto \V_\pm (\sigma)$ is a bounded operator from $L^q$ to $W^{1,q}_{loc}$ for all $q \in (1,\infty)$. Then $\V_\pm^\eps \overset{*}{\rightharpoonup} \V_\pm(\sigma)$ in $L^\infty([0,T]; W^{1,q}_{loc})$. Since we also have strong convergence $\V^\eps_\pm \to \V_\pm$ in $L^\infty([0,T]; L^q_{loc})$ for arbitrary $q \in (1,\infty)$. Then we must have $\V_\pm = \V_\pm(\sigma)\in L^\infty([0,T]; W^{1,q}_{loc})$. The above facts suffice to guarantee that $\sigma_\pm$ verifies \eqref{eq:continuity} in $L^\infty([0,T]; W^{-1,p})$ sense. 
    Since $\sigma^\eps_\pm$ verifies \eqref{eq:continuity} strongly (thus weakly), we observe that for any $\eta \in W^{1,p_*}$ with compact support, there is:
    \begin{equation}\label{est:dtsigma}
    \left|\int_{\R^2} \pa_t \sigma_{\pm}^\eps \eta dx\right| = \left|\int_{\R^2} \sigma_\pm^\eps \V_\pm^\eps \cdot \na \eta dx\right| \le C\|\tau\|_{L^\infty\cap L^1}^2 \|\eta\|_{W^{1,p_*}},
    \end{equation}
    which implies that $\pa_t \sigma_\pm^\eps$ is uniformly bounded in $L^\infty([0,T]; W^{-1,p})$, with $1/p + 1/p_* = 1$. After passing to the limit, the continuity with respect to time follows immediately from $\pa_t \sigma_\pm \in L^\infty([0,T]; W^{-1,p})$.

    \item \textbf{Additional Regularity of $\V$ and corresponding Lagrangian.} The Log-Lipschitz property of $\V$ follows from the potential-theoretical bound below:
    \be \la{est:loglipV}
        |\V_\pm(x,t) - \V_\pm (y,t)| \le C\|\tau_\pm\|_{L^1 \cap L^\infty} |x-y|(1 - \min\{0, \log|x-y|\}),\quad x,y \in \R^2.
    \ee
    Again, the above bound indeed holds since the Bessel potential only differs from the Green's function far from the diagonal set. The bound \eqref{est:holderX} follows from \eqref{est:loglipV} and a standard ODE argument.

\end{enumerate}
We then proceed to prove uniqueness by using the pseudo-energy method à la Yudovich \cite{yudovich}. Take two solutions $\omega^{(i)}$ and $F^{(i)}$, $i=1,2$ and denote
\[
\Omega = \fr{1}{2}( \omega^{(1)} + \omega^{(2)}), \quad F = \fr{1}{2}( F^{(1)} + F^{(2)}),
\]
and
\[
\omega = \fr{1}{2}( \omega^{(1)} - \omega^{(2)}), \quad f = \fr{1}{2}( F^{(1)} - F^{(2)}).
\]
Taking differences, the system \eqref{origeq}  with $\epsilon=1$ becomes
\be
\left\{
\ba
\pa_t f +\{ \Psi, f\} + \{\psi, F\} + \{\phi, \Omega\} + \{\Phi, \omega\} = 0,\\
\pa_t \omega + \{ \Psi, \omega\} + \{\psi, \Omega\} + \{\phi, F\} + \{\Phi, f\} = 0.
\ea
\right.
\la{difforig}
\ee

We multiply the first equation by $\phi$, the second by $-\psi$ and integrate in space by parts:
\begin{align*}
    \frac12 \frac{d}{dt} \left(\|\na\phi\|_{L^2}^2 + \|\phi\|_{L^2}^2\right) &= \int_{\R^2} \nabla^\perp \Psi \cdot \na\phi (-\Delta\phi + \phi) dx + \int_{\R^2} \na^\perp \psi \cdot \na \phi F dx + \int_{\R^2} \na^\perp \Phi \cdot \na\phi \omega dx\\
    &= -\int \na^\perp \Psi \cdot \na\phi(\Delta\phi) dx +\int_{\R^2} \na^\perp \psi \cdot \na \phi F dx + \int_{\R^2} \na^\perp \Phi \cdot \na\phi \omega dx\\
    &= \int_{\R^2} \pa_k (\na^\perp\Psi)_j\pa_j\phi \pa_k\phi dx +\int_{\R^2} \na^\perp \psi \cdot \na \phi F dx + \int_{\R^2} \na^\perp \Phi \cdot \na\phi \omega dx,\\
    \frac12\frac{d}{dt} \|\na \psi\|_{L^2}^2 &= -\int_{\R^2} \na^\perp \Psi \cdot \na \psi (\Delta \psi) dx -\int_{\R^2} \na^\perp \Phi \cdot \na\psi f dx-\int_{\R^2}\na^\perp \phi \cdot \na \psi F dx\\
    &=\int_{\R^2} \pa_k (\na^\perp\Psi)_j\pa_j\psi \pa_k\psi dx-\int_{\R^2} \na^\perp \Phi \cdot \na\psi f dx+\int_{\R^2}\na^\perp \psi \cdot \na \phi F dx.
\end{align*}
Summing the above two identities, we obtain
\begin{equation}
    \label{sumenergy}
    \begin{split}
    \frac12 \frac{d}{dt}\left(\|\na \psi\|_{L^2}^2+\|\na\phi\|_{L^2}^2 + \|\phi\|_{L^2}^2\right) &= \int_{\R^2} \pa_k (\na^\perp\Psi)_j (\pa_j\psi \pa_k\psi + \pa_j\phi \pa_k\phi) dx \\
    &\quad+ 2\int_{\R^2}(\na^\perp \psi)_j \pa_j \phi F dx + \int_{\R^2} (\na^\perp \Phi)_j (\pa_j\phi \omega - \pa_j\psi f).
    \end{split}
\end{equation}

We focus on the third term on the right-hand-side of \eqref{sumenergy}. The term multiplying $(\na^{\perp} \Phi)_j$ is
\[
\D \psi \pa_j\phi + (\D\phi -\phi)\pa_j\psi = \pa_k(\pa_k\psi \pa_j \phi + \pa_k\phi \pa_j\psi) -\pa_j(\pa_k\phi\pa_k\psi) - \phi \pa_j\psi
\]
We plug the above identity into \eqref{sumenergy} and integrate by parts to obtain
\begin{equation}
    \label{sumenergy2}
    \begin{split}
    \frac12 \frac{d}{dt}\left(\|\na \psi\|_{L^2}^2+\|\na\phi\|_{L^2}^2 + \|\phi\|_{L^2}^2\right) &= \int_{\R^2} \pa_k (\na^\perp\Psi)_j (\pa_j\psi \pa_k\psi + \pa_j\phi \pa_k\phi) dx \\
    &\quad- \int_{\R^2} \pa_k(\na^\perp \Phi)_j(\pa_k\psi \pa_j\phi + \pa_k\phi\pa_j\psi) dx\\
    &\quad+ 2\int_{\R^2}\na^\perp \psi \cdot \na \phi F dx - \int_{\R^2}\na^\perp \Phi\cdot \na \psi \phi dx.
    \end{split}
\end{equation}
Now we recall that 
$
-\Delta \Phi + \Phi = F,\Delta \Psi = \Omega.
$
Then the Calder\'on-Zygmund estimate informs us that for $p \in (1,\infty)$ sufficiently large,
\begin{equation}
    \label{est:LpCZ}
    \|\na^2\Phi\|_{L^p} \lesssim p \|F\|_{L^p} \lesssim p\|F_0\|_{L^p},\quad \|\na^2 \Psi\|_{L^p} \lesssim p\|\Omega\|_{L^p} \lesssim p\|\Omega_0\|_{L^p},
\end{equation}
where we also used the fact that $L^p$ norms of $\omega^{(i)}$, $F^{(i)}$, $i = 1,2$, are a priori bounded. Moreover, we have the following gradient estimates:
\begin{equation}
    \label{est:gradientest}
    \|\na \psi\|_{L^\infty} + \|\na \phi\|_{L^\infty} \le m, \|\na \Phi\|_{L^\infty} \le M,
\end{equation}
where $m, M$ depend only on the $L^1\cap L^\infty$ norm of the initial data. 

Define $E(t) = (\|\na \psi\|_{L^2}^2+\|\na\phi\|_{L^2}^2 + \|\phi\|_{L^2}^2)^\frac12$ and fix $\epsilon > 0$ sufficiently small. Using \eqref{sumenergy2}, \eqref{est:gradientest}, and H\"older inequality, we have
\begin{equation}\label{est:energy}
\begin{split}
    E\frac{dE}{dt} &\le m^\epsilon (\|\na^2 \Psi\|_{L^{2/\epsilon}} + \|\na^2 \Phi\|_{L^{2/\epsilon}})E^{2-\epsilon} + 2(\|F\|_{L^\infty} + \|\na \Phi\|_{L^\infty}) E^2\\
    &\le C \left(m^\epsilon \left(\frac{2}{\epsilon}M_0\right) E^{2-\epsilon} + M E^2\right),
\end{split}
\end{equation}
where $m, M, M_0$ are constants which only depend on the initial data. Here, we also used the \textit{a priori} bound $\|F\|_{L^\infty} \lesssim \|\tau\|_{L^\infty}$.

We now claim that given $\epsilon$ sufficiently small and $t \le \frac{1}{8CM_0} =: t_0$, the following bound holds:
\begin{equation}
    \label{est:bootstrap}
    E(t)^\epsilon \le 4Cm^\epsilon M_0t.
\end{equation}
We prove this claim by a bootstrap argument. It is clear that \eqref{est:bootstrap} holds on $[0,\delta]$ for some $\delta$ sufficiently small thanks to $E(0) = 0$. Assuming \eqref{est:bootstrap}, we have $E(t)^\epsilon \le \frac{m^\epsilon}{2}$ on $[0,t_0]$ and therefore
$$
\epsilon E^\epsilon m^{-\epsilon} \le \frac\epsilon2 \le \frac{M_0}{M},
$$
after choosing $\epsilon$ sufficiently small. Rearranging the above inequality, we arrive at
$$
m^\epsilon \left(\frac{1}{\epsilon}M_0\right) E^{-\epsilon} > M,
$$
which gives the following after combining with \eqref{est:energy}:
$$
\frac{dE}{dt} \le C\left(m^\epsilon \left(\frac{2}{\epsilon}M_0\right) E^{1-\epsilon} + M E\right) \le Cm^\epsilon \frac3\epsilon M_0 E^{1-\epsilon},\quad t \in [0,t_0].
$$
Integrating the above inequality with respect to $t$ yields
$$
E(t)^\epsilon \le 3Cm^\epsilon M_0 t,\quad t \in [0,t_0],
$$
which proves the claim. Now for $t \in [0,t_0]$, we take $\epsilon \to 0$ in \eqref{est:bootstrap} and conclude $E(t) = 0$ on $t \in [0,t_0]$. The uniqueness for all $t > 0$ follows from iterating the same procedure on $[kt_0, (k+1)t_0]$ for all $k \in \mathbb{N}$.

\end{proof}
 
\section{Merger of patches for the left-handed system} \la{patchmerge}
Patches are solutions which have $F\pm\omega$ constant on domains. The assumptions are that the initial data obey $(F\pm\omega)(x,0) = \mathbf 1_{\Omega_0^{\pm}}(x)$, the characteristic functions of bounded sets $\Omega_0^{\pm}$ in $\Rr^2$ with smooth boundaries $\Gamma_0^{\pm} = \pa\Omega_0^{\pm}$. Then it follows from the results of the previous section that $(F\pm\omega)(x,t) = \mathbf 1_{\Omega_t^{\pm}}(x)$ where $\Omega_t^{\pm}$ is the image of $\Omega_0^{\pm}$ under the flow with stream function $\psi \pm \phi$. 
 If the patch boundaries $\Gamma_t^{\pm}$ are initially smooth, then they remain smooth, by analogy with the well known situation for fluids (\cite{bc},\cite{chemin}), as long as the patches do not intersect. However, because the two regular flows are different, nothing prevents the merger of the two patches.  The change of topology of the streamlines of $\phi$ which are magnetic lines is due to the fact that they are not carried by the flow of electrons.

In this section, we consider the system \eqref{active} with the left-handed equation of state \eqref{pmleft}. We start with the following identities in view of the left-handed equation of state \eqref{pmleft} and the definitions of operators $\mathbb{K}, \mathbb{L}, \mathbb{U},$ and $\mathbb{B}$:
$$
v_+ = \U(\omega) + \B(F) = \U(\sigma_+) + \B(F) - \U(F) = \U(\sigma_+) - 2\na^\perp \mathbb{K}(F),
$$
$$
v_- = \U(\omega) - \B(F) = -\U(\sigma_-) - (\B(F) - \U(F)) = -\U(\sigma_-) + 2\na^\perp \mathbb{K}(F).
$$
With the above identities, we rewrite \eqref{sigmaeq} as
\begin{equation}
    \label{sigmaeq2}
    \begin{cases}
    \pa_t \sigma_+ + \U(\sigma_+) \cdot \nabla \sigma_+ -2\na^\perp \mathbb{K}(F)\cdot \nabla \sigma_+ = 0,\\
    \pa_t \sigma_- - \U(\sigma_-) \cdot \nabla \sigma_- +2\na^\perp \mathbb{K}(F)\cdot \nabla \sigma_- = 0.
    \end{cases}
\end{equation}

\subsection{The unscreened equations} \label{subsect:lefthandedunscreen}
As mentioned in the introduction, the left-handed system \eqref{sigmaeq2} is dominated by the following leading-order system at large frequencies:
\begin{equation}
    \label{reducedeq}
    \pa_t \sigma_\pm \pm \U(\sigma_\pm)\cdot \nabla \sigma_\pm = 0,
\end{equation}
In this subsection, we consider a merger of the above leading-order system.

The leading-order system \eqref{reducedeq} consists of two independent incompressible Euler equations in two dimensions, with the second one having a different convention of rotation. In particular, both equations in \eqref{reducedeq} admit circular and elliptical uniformly rotating patches. We consider the initial data
$
\tau_\pm(x) = \mathbf 1_{\Omega_0^\pm }(x),
$
where 
\begin{equation}\label{data}
\Omega_0^+ = \{(x,y) \in \R^2\; |\; \frac{(x-d-2R)^2}{R^2} + \frac{y^2}{(2R)^2} < 1\},\quad \Omega_0^- = \{(x,y) \in \R^2\;|\; \frac{x^2 + y^2}{R^2} < 1\},
\end{equation}
where $R > 0$, $d \in (0,R)$ are parameters to be determined. Identifying $\R^2$ with $\mathbb{C}$ and writing boundaries $\Gamma_0^\pm = \pa \Omega_0^\pm$ to be parametrized by $z_\pm^0(\alpha)$ with $z_\pm^0(0) = z_\pm^0 (2\pi)$, we have
$$
z_+^0(\alpha) = R\left[e^{i\frac\pi2}\left(\frac{3}{2}e^{i\alpha} + \frac{1}{2}e^{-i\alpha}\right) + \left(2+\frac{d}{R}\right)\right],\quad z_-^0(\alpha) = Re^{i\alpha}.
$$
Denoting the parametrization at time $t$ as $z_\pm(\alpha,t)$, from Euler dynamics we obtain the following explicit formulas (see e.g. \cite{farfield}):
\begin{equation}\label{background}
\begin{aligned}
z_+(\alpha,t) &= e^{i\frac\pi2}\left(\frac{3R}{2}\exp\left({i(\alpha-\frac{t}{2\pi}\frac{A}{\frac94 R^2})}\right) + \frac{R}{2}e^{-i\alpha}\right) + (d+2R)\\
&= \frac{R}{2}e^{i\frac{\pi}{2}}\left(3e^{-i\frac49 t}e^{i\alpha} + e^{-i\alpha}\right)+ (d+2R),\\
z_-(\alpha,t) &= Re^{-\frac{i}{2}t}e^{i\alpha},
\end{aligned}
\end{equation}
where in the computation above, $A$ is the area of the ellipse, and we have
$$
\frac{A}{\frac94 R^2} = \frac{\pi\left(\frac94 R^2 - \frac14 R^2\right)}{\frac94 R^2} = \frac89 \pi.
$$
In particular, the above formulas inform us that the patch $\Omega_t^+$ rigidly rotates counterclockwisely with angular velocity $\frac49$, and it takes $T_0 = \frac{9\pi}{4}$ to rotate by $\pi$. Let us denote $T_* = T_*(d,R)$ be the first time such that $\overline{\Omega_t^+} \cap\overline{\Omega_t^-} \neq \emptyset$. It is clear that $T_*$ is well defined whenever $d \in (0,R)$. Moreover,
$$
\lim_{d \to 0} T_*(d,R) = 0,\quad \lim_{d \to R} T_*(d,R) = \frac{9\pi}{8}.
$$
By elementary planar geometry, we observe that ${\Omega_t^+} \cap {\Omega_t^-} \neq \emptyset$ for all $t \in (T_*, \frac{9\pi}{4} - T_*)$, and ${\Omega_t^+} \cap {\Omega_t^-} = \emptyset$ otherwise. Thus, for the patch solution to the leading order model \eqref{reducedeq} with initial data $\tau_\pm = \mathbf 1_{\Omega_0^\pm }(x)$ given in \eqref{data}, we already obtain a merger of patches. Moreover, in view of the following identity:
\begin{equation}\label{defn:Fpatch}
F(x,t) = \begin{cases}
    \frac12,& x \in {\Omega_t^+} \Delta {\Omega_t^-},\\
    1,& x \in {\Omega_t^+} \cap {\Omega_t^-},
\end{cases}
\end{equation}
we observe  the topology change for the $F$-patches in the following sense: $\{F(\cdot,t) = 1\} \neq \emptyset$ for all $t \in (T_*, \frac{9\pi}{4} - T_*)$, and $\{F(\cdot,t) = 1\}= \emptyset$ otherwise.
\subsection{Merger for the screened equations}
In the screened left-handed equations \eqref{sigmaeq2}, in view of the operator $\mathbb{K}$ defined in \eqref{mathbbK}, we regard the term $\nabla^\perp \mathbb K(F)$ as a small perturbation at sufficiently small scales. We recall that the kernel for the operator $(\Delta - \mathbb{I})^{-1}$ is the Bessel potential $G_B(r) = -\frac{1}{2\pi}K_0(r)$, where $K_0$ is the modified Bessel function of the second kind. Writing
$$
\bar G(r) = -K_0(r) -  \log r,
$$
by definition of $F$, linearity of operators $\B, \U$, and an application of divergence theorem, we have
\begin{align*}
-2\nabla^\perp \mathbb K(F) &=\B(F) - \U(F) = \frac{1}{2}(\B(\sigma_+) - \U(\sigma_+)) + \frac12 \left(\B(\sigma_-) - \U(\sigma_-)\right)\\
&= \frac{1}{4\pi}\nabla^\perp\bigg(\int_{\Omega^+_t} \bar G(|x-y|) dy+ \int_{\Omega^-_t} \bar G(|x-y|) dy\bigg)\\
&= \frac{1}{4\pi} \int_0^{2\pi} \bar G(|x - z_+(\beta)|)\frac{\pa z_+}{\pa \beta}(\beta) d\beta + \frac{1}{4\pi} \int_0^{2\pi} \bar G(|x - z_-(\beta)|)\frac{\pa z_-}{\pa \beta}(\beta) d\beta.
\end{align*}
Therefore, the contour dynamics for the full problem \eqref{sigmaeq2} is written as follows:
\begin{equation}
\label{contour}
\begin{cases}
    \pa_t z_+ = \frac{1}{2\pi}\int_0^{2\pi} \log|z_+(\alpha) - z_+(\beta)| \frac{\pa z_+}{\pa \beta}(\beta) d\beta\\
    \quad+ \frac{1}{4\pi} \int_0^{2\pi} \bar G(|z_+(\alpha) - z_+(\beta)|)\frac{\pa z_+}{\pa \beta}(\beta) d\beta+ \frac{1}{4\pi} \int_0^{2\pi} \bar G(|z_+(\alpha) - z_-(\beta)|)\frac{\pa z_-}{\pa \beta}(\beta) d\beta,\\
    \pa_t z_- = -\frac{1}{2\pi}\int_0^{2\pi} \log|z_-(\alpha) - z_-(\beta)| \frac{\pa z_-}{\pa \beta}(\beta) d\beta\\
    \quad- \frac{1}{4\pi} \int_0^{2\pi} \bar G(|z_-(\alpha) - z_+(\beta)|)\frac{\pa z_+}{\pa \beta}(\beta) d\beta- \frac{1}{4\pi} \int_0^{2\pi} \bar G(|z_-(\alpha) - z_-(\beta)|)\frac{\pa z_-}{\pa \beta}(\beta) d\beta.
\end{cases}
\end{equation}
An observation is the following expansion for the modified Bessel function (see e.g. \cite{bessel,bessel1}): there exists $r_0 > 0$ such that for $ r \in (0,r_0)$:
\begin{equation}\label{tildeG}
K_0(r) = -\left(\log r - \log 2 + \gamma_0\right) + \tilde G(r),
\end{equation}
where $\gamma_0$ is the Euler's constant and $|\tilde G(s)| \lesssim r^2 |\log r|$. Moreover, $\tilde G(r)$ can be extended to $r= 0$ by $\tilde G(0) = 0$ as an $C^1([0,\infty))$ function, and $\tilde G'(r)$ is Log-Lipschitz near the origin. With the aforementioned properties, $\bar G(r)$ is rewritten as follows:
$$
\bar G(r) = (\log 2 - \gamma_0) - \tilde G(r),\quad r \in (0,r_0)
$$
Hence, if we assume that the scale for the patches, namely $R$, is sufficiently small, we have
\begin{align*}
    \int_0^{2\pi} \bar G(|z_\pm(\alpha) - z_\pm(\beta)|)\frac{\pa z_\pm}{\pa \beta}(\beta) d\beta &= \int_0^{2\pi} \left((\log 2 - \gamma_0) - \tilde G(|z_\pm(\alpha) - z_\pm(\beta)|)\right)\frac{\pa z_\pm}{\pa \beta}(\beta) d\beta\\
    &= -\int_0^{2\pi} \tilde G(|z_\pm(\alpha) - z_\pm(\beta)|)\frac{\pa z_\pm}{\pa \beta}(\beta) d\beta,
\end{align*}
where we used periodicity of $z_\pm$ in the second equality above. Thus, we reduce \eqref{contour} to:
\begin{equation}
    \label{contour2}
    \begin{cases}
        \pa_t z_+ = \frac{1}{2\pi}\int_0^{2\pi} \log|z_+(\alpha) - z_+(\beta)| \frac{\pa z_+}{\pa \beta}(\beta) d\beta\\
    \quad- \frac{1}{4\pi} \int_0^{2\pi} \tilde G(|z_+(\alpha) - z_+(\beta)|)\frac{\pa z_+}{\pa \beta}(\beta) d\beta - \frac{1}{4\pi} \int_0^{2\pi} \tilde G(|z_+(\alpha) - z_-(\beta)|)\frac{\pa z_-}{\pa \beta}(\beta) d\beta,\\
    \pa_t z_- = -\frac{1}{2\pi}\int_0^{2\pi} \log|z_-(\alpha) - z_-(\beta)| \frac{\pa z_-}{\pa \beta}(\beta) d\beta\\
    \quad+ \frac{1}{4\pi} \int_0^{2\pi} \tilde G(|z_-(\alpha) - z_+(\beta)|)\frac{\pa z_+}{\pa \beta}(\beta) d\beta + \frac{1}{4\pi} \int_0^{2\pi} \tilde G(|z_-(\alpha) - z_-(\beta)|)\frac{\pa z_-}{\pa \beta}(\beta) d\beta.
    \end{cases}
\end{equation}
The terms involving $\tilde G$ are perturbative because $\tilde G$ is not a singular kernel and has small size at a sufficiently small scale $R$.

Then, the route to prove magnetic reconnection is to study \eqref{contour2} with initial data $\sigma_\pm(x,0) = \mathbf{1}_{\Omega_0^\pm}(x)$, where $\Omega_0^\pm$ is given in \eqref{data}. Denoting the perturbation by $\ze_\pm = \frac{z_\pm - Z_\pm}{R}$, where $Z_\pm$ is the background given by \eqref{background}, we consider the equations satisfied by $\ze_\pm$ equipped with zero initial data. We then prove a long-time stability statement where $\|\ze_\pm\|_{C^{1,\gamma}} \ll 1$ over a sufficiently long time interval.

Let us fix the background $Z_\pm(\alpha, t)$ by the following formulas:
\begin{equation}\label{background2}
\begin{aligned}
Z_+(\alpha,t) &= \frac{R}{2}e^{i\frac{\pi}{2}}\left(3e^{-i\frac49 t}e^{i\alpha} + e^{-i\alpha}\right)+ (d+2R),\\
Z_-(\alpha,t) &= Re^{-\frac{i}{2}t}e^{i\alpha}.
\end{aligned}
\end{equation}
where $d \in (\frac{R}{4}, \frac{3R}{4})$. In particular, they solve the contour dynamics equation for 2D Euler:
$$
\pt Z_\pm = \pm\frac{1}{2\pi}\int_0^{2\pi} \log|Z_\pm(\alpha) - Z_\pm(\beta)| \frac{\pa Z_\pm}{\pa \beta}(\beta) d\beta.
$$
Using the above equation, \eqref{background2}, \eqref{contour2}, and properties of logarithm, a computation yields the following system satisfied by perturbation $\ze_\pm$:
\begin{equation}
    \label{contourper}
    \begin{aligned}
    \pt \ze_+ &= \frac12 \calH \ze_+ + \frac{1}{2\pi}\int_0^{2\pi}\log\left|3e^{-i\frac49 t} - e^{-i(\alpha + \beta)}\right|\pb \ze_+(\beta) d\beta\\
    &\quad + \frac{1}{2\pi R}\int_0^{2\pi}\log\left|1+\frac{2e^{-i\frac\pi2}}{3e^{-i\frac49 t} - e^{-i(\alpha + \beta)}}\frac{\ze_+(\alpha)-\ze_+(\beta)}{e^{i\alpha} - e^{i\beta}}\right|\pb Z_+(\beta) d\beta\\
    &\quad + \frac{1}{2\pi}\int_0^{2\pi}\log\left|1+\frac{2e^{-i\frac\pi2}}{3e^{-i\frac49 t} - e^{-i(\alpha + \beta)}}\frac{\ze_+(\alpha)-\ze_+(\beta)}{e^{i\alpha} - e^{i\beta}}\right|\pb \ze_+(\beta) d\beta\\
    &\quad - \sum_{\pm} \frac{1}{4\pi R} \int_0^{2\pi} \tilde G(|Z_+(\alpha) - Z_\pm(\beta) + R(\ze_+(\alpha) - \ze_\pm(\beta))|) \pb Z_\pm(\beta) d\beta\\
    &\quad - \sum_{\pm} \frac{1}{4\pi} \int_0^{2\pi} \tilde G(|Z_+(\alpha) - Z_\pm(\beta) + R(\ze_+(\alpha) - \ze_\pm(\beta))|) \pb \ze_\pm(\beta) d\beta,\\
    \pt \ze_- &= -\frac12 \calH \ze_- - \frac{1}{2\pi R}\int_0^{2\pi} \log\left|1+ e^{\frac{i}{2} t} \frac{\ze_-(\alpha) - \ze_-(\beta)}{e^{i\alpha} - e^{i\beta}}\right| \pb Z_-(\beta) d\beta\\
    &\quad - \frac{1}{2\pi}\int_0^{2\pi}\log\left|1+ e^{\frac{i}{2} t} \frac{\ze_-(\alpha) - \ze_-(\beta)}{e^{i\alpha} - e^{i\beta}}\right| \pb \ze_-(\beta) d\beta\\
    &\quad + \sum_{\pm} \frac{1}{4\pi R} \int_0^{2\pi} \tilde G(|Z_-(\alpha) - Z_\pm(\beta) + R(\ze_-(\alpha) - \ze_\pm(\beta))|) \pb Z_\pm(\beta) d\beta\\
    &\quad + \sum_{\pm} \frac{1}{4\pi} \int_0^{2\pi} \tilde G(|Z_-(\alpha) - Z_\pm(\beta) + R(\ze_-(\alpha) - \ze_\pm(\beta))|) \pb \ze_\pm(\beta) d\beta,
    \end{aligned}
\end{equation}
Here, we also used the following classical identity (see e.g. \cite{farfield}):
$$
\frac{1}{2\pi}\int_0^{2\pi} \log|e^{i\alpha} - e^{i\beta}| \pb f(\beta) d\beta = \frac12 \calH f(\alpha),
$$
where $\calH$ denotes the Hilbert transform on torus. We recall that we also have the initial data
$$
\ze_\pm(\alpha,0) = 0.
$$

The well-posedness for \eqref{contourper} in $C^{1,\gamma}$, $\gamma \in (0,1)$, follows after establishing a priori H\"older estimates. We start by proving a few useful technical lemmas, the first and second of which concern integral operators with logarithmic kernels.
\begin{lemma}\label{lem:logintegral}
    Given $\gamma \in (0,1)$, assume $f, g \in C^{1,\gamma}$. Then there exists $\eps_0 > 0$ sufficiently small that for $\|f\|_{C^{1,\gamma}} < \eps_0$, the function
    \begin{equation}
        \label{eq:logintegral}
        I(\alpha) := \frac{1}{2\pi}\int_0^{2\pi} \log \left|1 + e^{\frac{i}{2}t}\frac{f(\alpha) - f(\beta)}{e^{i\alpha} - e^{i\beta}}\right| \pb g(\beta) d\beta
    \end{equation}
    is $C^{1,\gamma}$. Moreover, the following estimate holds:
    \begin{equation}
        \label{est:logintegral}
        \|I\|_{C^{1,\gamma}} \le C(\gamma)\|f\|_{C^{1,\gamma}}\|g\|_{C^{1,\gamma}}.
    \end{equation}
\end{lemma}
\begin{proof}
    For the rest of the proof, we use the notation $f'$ or $g'$ to denote derivative with respect to the parametrization parameter $\alpha$ or $\beta$. 
    Define $w(\alpha,\beta) := \frac{f(\alpha) - f(\beta)}{e^{i\alpha} - e^{i\beta}}$. Then we observe that for any $\alpha,\beta \in [0,2\pi)$:
    $$
    |w(\alpha, \beta)| \le C\|f\|_{C^1} \le C\eps_0.
    $$
    By choosing $\eps_0$ small, we have that
    \begin{equation}\label{est:nondeg}
        |1+ e^{\frac{i}{2}t} w(\alpha,\beta)| \ge 1 - \sup_{\alpha,\beta}|w(\alpha,\beta)| > \frac12.
    \end{equation}
    This implies $\log|1+ e^{\frac{i}{2}t} w(\alpha,\beta)|$ is uniformly bounded. Moreover by concavity of logarithm, we have the refined bound
    $$
    \log |1+ e^{\frac{i}{2}t} w(\alpha,\beta)| \le |w(\alpha,\beta)| \le C\|f\|_{C^1},
    $$
    from which we can  bound:
    \begin{equation}
        \label{est:logintegralLinfty}
        \|I\|_{L^\infty} \le C\|f\|_{C^1}\|g\|_{C^1}.
    \end{equation}
    Differentiating \eqref{eq:logintegral}, we observe that
    \begin{equation}
        \label{eq:paI}
        I'(\alpha) = \frac{1}{2\pi}\int_0^{2\pi}Re\left(\frac{e^{\frac{i}{2}t} \palpha w(\alpha,\beta)}{1+ e^{\frac{i}{2}t} w(\alpha,\beta)}\right) g'(\beta) d\beta.
    \end{equation}
    We first investigate the boundedness of $\palpha I$: via a computation, it holds that
    \begin{equation}
        \label{paw}
        \palpha w(\alpha,\beta) = \frac{1}{e^{i\alpha} - e^{i\beta}}\left(f'(\alpha) - \frac{f(\alpha) - f(\beta)}{e^{i\alpha} - e^{i\beta}}(ie^{i\alpha})\right).
    \end{equation}
    Denoting $h := \alpha - \beta$, an application of Mean Value Theorem yields:
    \begin{align*}
        f(\alpha) - f(\beta) = f'(\alpha) h + R_1(\alpha,\beta),\quad e^{i\alpha} - e^{i\beta} = ie^{i\alpha}h + R_2(\alpha,\beta),
    \end{align*}
    where the remainders $R_{j}$ verify the bounds
    $$
    |R_1(\alpha,\beta)| \le C\|f\|_{C^{1,\gamma}}|h|^{1+\gamma},\quad |R_2(\alpha,\beta)| \le C |h|^2.
    $$
    With the above bounds, we observe that
    $$
    \frac{f(\alpha) - f(\beta)}{e^{i\alpha} - e^{i\beta}} = \frac{f'(\alpha)}{ie^{i\alpha}} + \frac{ie^{i\alpha}R_1 - f'(\alpha) R_2}{ie^{i\alpha}(ie^{i\alpha} h + R_2)}=: \frac{f'(\alpha)}{ie^{i\alpha}} + R_3(\alpha,\beta), 
    $$
    with $R_3 \in C^\gamma$ such that $|R_3| \le C\|f\|_{C^{1,\gamma}}|h|^\gamma$. Plugging in \eqref{paw},
    \begin{equation}
        \label{paw1}
        \begin{aligned}
            \palpha w &= \frac{1}{e^{i\alpha} - e^{i\beta}}\left(ie^{i\alpha} R_3(\alpha,\beta)\right) = \frac{R_3(\alpha,\beta)}{h} + R_4(\alpha,\beta),
        \end{aligned}
    \end{equation}
    where $|R_4(\alpha,\beta)| \le C\|f\|_{C^{1,\gamma}}|h|^\gamma$. Thus, we conclude that 
    \begin{equation}
        \label{est:paw}
        |\p_\alpha w(\alpha,\beta)| \le C\|f\|_{C^{1,\gamma}}\min(|\alpha - \beta|^{\gamma - 1}, 1).
    \end{equation}
    Now recalling that $|1 + e^{\frac{i}{2}t} w| \ge \frac12$, we then have
    \begin{equation}
        \label{est:logintegralC1}
        \begin{aligned}
        |I'(\alpha)| &\le C\|f\|_{C^{1,\gamma}}\int_0^{2\pi} \min(|\alpha - \beta|^{\gamma - 1}, 1) |g'(\beta)| d\beta \\
        &\le C\|f\|_{C^{1,\gamma}}\|g\|_{C^1}\int_{|\alpha - \beta| \le 1} |\alpha - \beta|^{\gamma - 1} d\beta + C\|f\|_{C^{1,\gamma}}\|g\|_{C^1}\\
        &\le C\|f\|_{C^{1,\gamma}}\|g\|_{C^1}.
        \end{aligned}
    \end{equation}
    for any $\alpha \in [0,2\pi)$.
    Finally, we show $I'(\alpha) \in C^\gamma$. Fixing $\alpha, \alpha' \in [0,2\pi)$ with $\alpha \neq \alpha'$, 
    \begin{equation}
        \label{eq:I'adiff}
        \begin{aligned}
        I'(\alpha) - I'(\alpha') &= \frac{1}{2\pi}\int_0^{2\pi}Re\left(\frac{e^{\frac{i}{2}t} \palpha w(\alpha,\beta)}{1+ e^{\frac{i}{2}t} w(\alpha,\beta)} - \frac{e^{\frac{i}{2}t} \palpha w(\alpha',\beta)}{1+ e^{\frac{i}{2}t} w(\alpha',\beta)}\right) g'(\beta) d\beta\\
        &=\frac{1}{2\pi}\left(\int_{|\alpha - \beta| \le 2\delta} + \int_{|\alpha - \beta| > 2\delta}\right)Re\left(e^{\frac{i}{2}t}(K(\alpha,\beta) - K(\alpha',\beta))\right) g'(\beta) d\beta\\
        &=: I_{n} + I_{f}
        \end{aligned}
    \end{equation}
    where $\delta := |\alpha - \alpha'|$ and
    $$
    K(\alpha,\beta) := \frac{\palpha w(\alpha,\beta)}{1+ e^{\frac{i}{2}t} w(\alpha,\beta)}.
    $$
    We show appropriate $C^\gamma$ bounds for $I_n$ and $I_f$ respectively. We only focus on the regime where $|\alpha - \alpha'|$ and $|\alpha-\beta|$ are sufficiently small, which corresponds to the most singular regime.
    
    To estimate $I_n$, we recall the non-degeneracy bound $|1 + e^{\frac{i}{2}t} w| \ge \frac12$, and estimate that:
    \begin{align*}
        |I_n| &\le C\int_{|\alpha - \beta| \le 2\delta} |\palpha w(\alpha,\beta)||g'(\beta)| d\beta + C\int_{|\alpha - \beta| \le 2\delta} |\palpha w(\alpha',\beta)||g'(\beta)| d\beta.
    \end{align*}
    The first integral above can be bounded in a similar fashion to \eqref{est:logintegralC1} after invoking \eqref{est:paw}:
    \begin{align*}
        \int_{|\alpha - \beta| \le 2\delta} |\palpha w(\alpha,\beta)||g'(\beta)| d\beta &\le C\|f\|_{C^{1,\gamma}}\|g\|_{C^1}\int_{|\alpha - \beta| \le 2\delta} |\alpha - \beta|^{\gamma - 1} d\beta\\
        &\le C(\gamma)\|f\|_{C^{1,\gamma}}\|g\|_{C^1}\delta^\gamma.
    \end{align*}
    To bound the second integral, we note that $|\alpha - \beta| \le 2\delta$ implies that $|\alpha' - \beta| \le 3\delta$. Hence, we have
    \begin{align*}
        \int_{|\alpha - \beta| \le 2\delta} |\palpha w(\alpha',\beta)||g'(\beta)| d\beta &\le \int_{|\alpha' - \beta| \le 3\delta} |\palpha w(\alpha',\beta)||g'(\beta)| d\beta\\
        &\le C(\gamma)\|f\|_{C^{1,\gamma}}\|g\|_{C^1}\delta^\gamma.
    \end{align*}
    We have therefore deduced the following bound:
    \begin{equation}
        \label{est:In}
        |I_n| \le C(\gamma )\|f\|_{C^{1,\gamma}}\|g\|_{C^1}\delta^\gamma.
    \end{equation}
    Now we turn to the estimate for $I_f$. We first observe the following decomposition concerning the difference $K(\alpha,\beta) - K(\alpha',\beta)$:
    \begin{align*}
        K(\alpha,\beta) - K(\alpha',\beta) &= \frac{\palpha w(\alpha,\beta) - \palpha w(\alpha',\beta)}{1+ e^{\frac{i}{2}t} w(\alpha,\beta)} + \palpha w(\alpha',\beta)\left(\frac{1}{1+ e^{\frac{i}{2}t}w(\alpha,\beta)} - \frac{1}{1+ e^{\frac{i}{2}t}w(\alpha',\beta)}\right)\\
        &= \frac{\palpha w(\alpha,\beta) - \palpha w(\alpha',\beta)}{1+ e^{\frac{i}{2}t} w(\alpha,\beta)} + e^{\frac{i}{2}t}\palpha w(\alpha',\beta)\frac{w(\alpha',\beta) - w(\alpha,\beta)}{(1+ e^{\frac{i}{2}t}w(\alpha,\beta))(1+ e^{\frac{i}{2}t}w(\alpha',\beta))}\\
        &=: T_1 + T_2.
    \end{align*}
    Using \eqref{paw}, the first term $T_1$ will be further decomposed as follows:
    \begin{align*}
        T_1 &= \frac{1}{1+ e^{\frac{i}{2}t}w(\alpha,\beta)}\left(\frac{f'(\alpha) - f'(\alpha')}{e^{i\alpha} - e^{i\beta}}\right) + \frac{1}{1+ e^{\frac{i}{2}t}w(\alpha,\beta)}\sum_{k=1}^3 T_{1k},
    \end{align*}
    where
    \begin{align*}
        T_{11} &=\frac{e^{i\alpha'} - e^{i\alpha}}{(e^{i\alpha} - e^{i\beta})(e^{i\alpha'} - e^{i\beta})}\left(f'(\alpha') - w(\alpha',\beta) (ie^{i\alpha'})\right),\\
        T_{12} &= \frac{ie^{i\alpha}(w(\alpha',\beta) - w(\alpha,\beta))}{e^{i\alpha}  -e^{i\beta}},\quad T_{13} = \frac{iw(\alpha',\beta)(e^{i\alpha'} - e^{i\alpha})}{e^{i\alpha}  -e^{i\beta}}
    \end{align*}

    
    We first show that $T_{1k}$, $k = 1,2,3$, and $T_2$ are harmless. To do this we use the following elementary bound: for any $\beta \in [0,2\pi)$,
    \begin{equation}
        \label{est:deltaw}
        |w(\alpha,\beta) - w(\alpha',\beta)| \le C\|f\|_{C^1} \delta |\alpha - \beta|^{\gamma-1}.
    \end{equation}
    This bound follows from the Fundamental Theorem of Calculus and \eqref{est:paw}:
    \begin{align*}
        |w(\alpha,\beta) - w(\alpha',\beta)| &\le \int_{\alpha'}^\alpha |\palpha w(\xi,\beta)| d\xi \le C\|f\|_{C^1}\delta|\alpha - \beta|^{\gamma - 1}.
    \end{align*}
    To bound $T_{11}$, we recall that an almost identical argument to that proving \eqref{est:paw} gives 
    \begin{equation}\label{est:pawaux1}
    |f'(\alpha) - w(\alpha,\beta) (ie^{i\alpha})| \le C\|f\|_{C^{1,\gamma}} |\alpha - \beta|^\gamma.
    \end{equation}
    Also using that $|e^{i\alpha} - e^{i\beta}|\sim |\alpha - \beta|$ when $|\alpha - \beta|$ sufficiently small, and that $|\alpha - \beta| \sim |\alpha' - \beta|$ given $|\alpha - \beta| > 2\delta$, we have
    \begin{equation}
        \label{est:T11}
        |T_{11}| \le C\|f\|_{C^{1}} \delta |\alpha - \beta|^{\gamma-2}.
    \end{equation}
    A similar reasoning also yields the following bounds:
    \begin{equation}
        \label{est:T123}
        |T_{12}| \le C\|f\|_{C^{1}} \delta |\alpha - \beta|^{\gamma-2},\quad |T_{13}| \le C\|f\|_{C^{1}} \delta|\alpha - \beta|^{-1}.
    \end{equation}
    For $T_2$, we use \eqref{est:nondeg}, \eqref{est:paw}, and \eqref{est:deltaw} to conclude that
    \begin{equation}
        \label{est:T2}
        |T_2| \le C\|f\|_{C^{1,\gamma}} \delta |\alpha - \beta|^{2\gamma-2}.
    \end{equation}
    With the above estimates, we write
    \begin{align*}
        I_f &= \frac{1}{2\pi}\int_{|\alpha - \beta| > 2\delta}Re\left(\frac{e^{\frac{i}{2}t}}{1+ e^{\frac{i}{2}t}w(\alpha,\beta)}\left(\frac{f'(\alpha) - f'(\alpha')}{e^{i\alpha} - e^{i\beta}}\right)\right)g'(\beta) d\beta,\\
        &\quad + \int_{|\alpha - \beta| > 2\delta}\bar{K}(\alpha,\alpha',\beta)g'(\beta) d\beta =: J_f^1 + J_f^2,
    \end{align*}
    where from \eqref{est:T11}, \eqref{est:T123}, and \eqref{est:T2} we can estimate $J_f^2$ by:
    \begin{equation}
        \label{est:Jf2}
        \begin{aligned}
            |J_f^2| &\le C\|f\|_{C^{1,\gamma}}\|g\|_{C^1}\int_{|\alpha - \beta| > 2\delta} \left(\delta |\alpha - \beta|^{\gamma-2} + \delta |\alpha - \beta|^{2\gamma-2} + \delta |\alpha - \beta|^{-1}\right) d\beta\\
            &\le C(\gamma)\|f\|_{C^{1,\gamma}}\|g\|_{C^1}(\delta^\gamma + \max\{\delta^{2\gamma},\delta|\log\delta|, \delta\} + \delta |\log \delta|)\\
            &\le C(\gamma)\|f\|_{C^{1,\gamma}}\|g\|_{C^1}\delta^\gamma.
        \end{aligned}
    \end{equation}
    Here, we also used that $\gamma \in (0,1).$ To study the main term $J_f^1$, we rewrite it as
    \begin{align*}
        J_f^1 &= \frac{1}{2\pi}\int_{|\alpha - \beta| > 2\delta} Re\left(\frac{e^{\frac{i}{2}t}}{1+ e^{\frac{i}{2}t}\bar{w}(\alpha)}\left(\frac{f'(\alpha) - f'(\alpha')}{e^{i\alpha} - e^{i\beta}}\right)\right)g'(\beta) d\beta\\
        &\quad + \frac{1}{2\pi}\int_{|\alpha - \beta| > 2\delta} Re\left(e^{\frac{i}{2}t}\left(\frac{1}{1+ e^{\frac{i}{2}t}w(\alpha,\beta)} - \frac{1}{1+ e^{\frac{i}{2}t}\bar{w}(\alpha)}\right)\left(\frac{f'(\alpha) - f'(\alpha')}{e^{i\alpha} - e^{i\beta}}\right)\right)g'(\beta) d\beta\\
        &=: J_f^{1,1} + J_f^{1,2},
    \end{align*}
    where $\bar{w}(\alpha) = \frac{f'(\alpha)}{ie^{i\alpha}}$. The term $J_{f}^{1,2}$ can be bounded using the following estimate:
    \begin{align*}
        \left|\frac{1}{1+ e^{\frac{i}{2}t}w(\alpha,\beta)} - \frac{1}{1+ e^{\frac{i}{2}t}\bar{w}(\alpha)}\right| &= \left|\frac{\bar w(\alpha) - w(\alpha,\beta)}{(1+ e^{\frac{i}{2}t}w(\alpha,\beta))(1+ e^{\frac{i}{2}t}\bar{w}(\alpha))}\right|\\
        &\le C\|f\|_{C^{1,\gamma}} |\alpha - \beta|^\gamma,
    \end{align*}
    where we used the non-degeneracy bound \eqref{est:nondeg} and \eqref{est:pawaux1}. With the above inequality, we have
    \begin{align*}
        |J_f^{1,2}| \le C\|f\|_{C^{1,\gamma}}^2\|g\|_{C^1}\int_{|\alpha - \beta| > 2\delta} \delta^\gamma|\alpha - \beta|^{\gamma - 1} d\beta \le C\eps_0\|f\|_{C^{1,\gamma}}\|g\|_{C^1}\delta^\gamma.
    \end{align*}
    To treat $J_f^{1,1}$, we note the following identity: letting $F(\alpha,\alpha',t) = \frac{e^{\frac{i}{2}t}}{1+ e^{\frac{i}{2}t}\bar{w}(\alpha)}(f'(\alpha) - f'(\alpha'))$,
    \begin{align*}
        Re\left(\frac{e^{\frac{i}{2}t}}{1+ e^{\frac{i}{2}t}\bar{w}(\alpha)}\left(\frac{f'(\alpha) - f'(\alpha')}{e^{i\alpha} - e^{i\beta}}\right)\right) &= \frac12 Re\left(F e^{-i\alpha}\right) + \frac12 Im\left(F e^{-i\alpha}\right)\cot\left(\frac{\alpha - \beta}{2}\right),
    \end{align*}
    where $F$ obeys the bound $|F(\alpha,\alpha',t)| \le C\|f\|_{C^{1,\gamma}} \delta^\gamma$. Writing
    \begin{align*}
        J_{f}^{1,1} &= \frac{1}{4\pi} \int_{|\alpha - \beta| > 2\delta} Re\left(F e^{-i\alpha}\right) g'(\beta) d\beta + \frac{Im\left(F e^{-i\alpha}\right)}{4\pi}\int_{|\alpha - \beta| > 2\delta}\cot\left(\frac{\alpha - \beta}{2}\right)g'(\beta) d\beta,
    \end{align*}
    The first integral can be absolutely bounded by $C\|f\|_{C^{1,\gamma}} \|g\|_{C^1}\delta^\gamma$. Using the oddness of kernel $\cot(z/2)$, the absolute value of the second integral equals to
    \be
    \ba
        \left|\frac{Im\left(F e^{-i\alpha}\right)}{4\pi}\int_{|\alpha - \beta| > 2\delta}\cot\left(\frac{\alpha - \beta}{2}\right)(g'(\beta)-g'(\alpha)) d\beta\right| \\ \le C\|f\|_{C^{1,\gamma}}\|g\|_{C^{1,\gamma}}\delta^\gamma\int_{|\alpha-\beta|>2\delta}|\alpha - \beta|^{\gamma - 1}d\beta\\
        \le C\|f\|_{C^{1,\gamma}}\|g\|_{C^{1,\gamma}}\delta^\gamma,
    \ea
    \ee
    where we used that $|\alpha - \beta|$ is sufficiently small. Hence, we have proved
    \begin{equation}
        \label{est:If}
        |I_f| \le C(\gamma)\|f\|_{C^{1,\gamma}}\|g\|_{C^{1,\gamma}}\delta^\gamma,
    \end{equation}
    which combining with \eqref{est:In} yields $\|I'(\alpha)\|_{C^\gamma} \le C(\gamma)\|f\|_{C^{1,\gamma}}\|g\|_{C^{1,\gamma}}$. This completes the proof.
\end{proof}

\begin{lemma}\label{lem:logintegral2}
    Given $\gamma \in (0,1)$, assume $f, g \in C^{1,\gamma}$. Then there exists $\eps_0 > 0$ sufficiently small that for $\|f\|_{C^{1,\gamma}} < \eps_2$, the function
    \begin{equation}
        \label{eq:logintegral2}
        I(\alpha) := \frac{1}{2\pi}\int_0^{2\pi} \log \left|1 + \frac{2e^{-i\frac\pi2}}{3e^{-i\frac49 t} - e^{-i(\alpha + \beta)}}\frac{f(\alpha) - f(\beta)}{e^{i\alpha} - e^{i\beta}}\right| \pb g(\beta) d\beta
    \end{equation}
    is $C^{1,\gamma}$. Moreover, the following estimate holds:
    \begin{equation}
        \label{est:logintegral2}
        \|I\|_{C^{1,\gamma}} \le C(\gamma)\|f\|_{C^{1,\gamma}}\|g\|_{C^{1,\gamma}}.
    \end{equation}
\end{lemma}
\begin{proof}
    The proof of this Lemma is essentially the same as that of Lemma \ref{lem:logintegral}. We observe that the function
    $$
    \frac{2e^{-i\frac\pi2}}{3e^{-i\frac49 t} - e^{-i(\alpha + \beta)}}
    $$
    is smooth in both $\alpha$ and $\beta$. Moreover, we have the bound
    $$
    \left|\frac{2e^{-i\frac\pi2}}{3e^{-i\frac49 t} - e^{-i(\alpha + \beta)}}\right| \le 1.
    $$
    The two facts above guarantee that the same argument for Lemma \ref{lem:logintegral} goes through with minor changes.
\end{proof}

The last technical lemma concerns an integral operator involving kernel $\tilde G$:
\begin{lemma}\label{lem:tGintegral}
    Given $\gamma \in (0,1)$, assume $f_1, f_2, g_1, g_2\in C^{1,\gamma}$ and $h \in C^1$. Then there exists $\eps_1 > 0$ sufficiently small that for $\|f_j\|_{C^{1,\gamma}}, \|g_j\|_{C^{1,\gamma}} < \eps_1$, $j = 1,2$, the function
    \begin{equation}
        \label{eq:tGintegral}
        \tilde I(\alpha) := \int_0^{2\pi} \Tilde{G}\left(|f_1(\alpha) - f_2(\beta) + g_1(\alpha) - g_2(\beta)|\right)  \pb h(\beta) d\beta
    \end{equation}
    is $C^{1,\gamma}$. Moreover, the following estimate holds:
    \begin{equation}
        \label{est:tGintegral}
        \|\tilde I\|_{C^{1,\gamma}} \le C\eps_1^2 |\log\eps_1| \|h\|_{C^1}.
    \end{equation}
\end{lemma}
\begin{proof}
For convenience, denote
$$
F(\alpha,\beta) := f_1(\alpha) - f_2(\beta) + g_1(\alpha) - g_2(\beta).
$$
Then by assumption, $F(\alpha,\beta) \le 4\eps_1 $ for any $\alpha,\beta \in [0,2\pi)$. Recall that for $\tilde G$ defined in \eqref{tildeG}, we have
$$
|\tilde G(r)| \le C r^2 |\log r|. 
$$
for $r > 0$ sufficiently small. Then by selecting $\eps_1$ sufficiently small, we have
\begin{equation}\label{est:tILinfty}
\|\tilde I\|_{L^\infty} \le C\eps_1^2 |\log\eps_1| \|h\|_{C^1}.
\end{equation}
Differentiating \eqref{eq:tGintegral}, we have
\begin{equation}
    \label{eq:patI}
    \tilde I'(\alpha) = \int_0^{2\pi} Re\left[\tilde G'(|F(\alpha,\beta)|)\frac{F(\alpha,
    \beta)}{|F(\alpha,\beta)|}(\bar{f}_1'(\alpha) + \bar{g}_1'(\alpha))\right]  h'(\beta) d\beta.
\end{equation}
Recall that $\tilde G'(r) \le Cr |\log r|$ when $r> 0$ is sufficiently small. Then it is clear from \eqref{eq:patI} that
\begin{equation}\label{est:tI'Linfty}
\|\tilde I'\|_{L^\infty} \le C\eps_1 |\log \eps_1| \cdot \eps_1 \|h\|_{C^1} \le C\eps_1^2 |\log\eps_1| \|h\|_{C^1}.
\end{equation}
Finally, we fix $\alpha \neq \alpha'$, and $\delta =: |\alpha - \alpha'|$. For notational convenience, let us write
\begin{equation}
    \label{eq:H}
    H(\alpha,\beta) := \tilde{G}'(|F(\alpha,\beta)|) \frac{F(\alpha,\beta)}{|F(\alpha,\beta)|}.
\end{equation}
We further without loss of generality assume that $\delta$ is sufficiently small. Then
\begin{equation}\label{eq:difftI}
\begin{aligned}
    \tilde I'(\alpha) - \tilde I'(\alpha') &= \int_0^{2\pi} Re\left[(H(\alpha,\beta) - H(\alpha',\beta))(\bar{f}_1'(\alpha) + \bar{g}_1'(\alpha))\right]h'(\beta) d\beta\\
    & \quad + \int_0^{2\pi} Re\left[H(\alpha',\beta)(\bar{f}_1'(\alpha) - \bar{f}_1'(\alpha') + \bar{g}_1'(\alpha)-\bar{g}_1'(\alpha'))\right]  h'(\beta) d\beta\\
    &=: J_1 + J_2.
\end{aligned}
\end{equation}
By the assumptions on norms of $f_1, g_1$, and the bound $\tilde G'(r) \le Cr |\log r|$ when $r> 0$ is sufficiently small, it is straightforward to obtain that
\begin{equation}
    \label{est:difftI2}
    |J_2| \le C\eps_1^2 |\log\eps_1| \|h\|_{C^1}\delta^\gamma.
\end{equation}
To estimate $J_1$, we observe that the bound
\begin{equation}
    \label{est:difftI1}
    |J_1| \le C\eps_1^2 |\log\eps_1| \|h\|_{C^1}\delta^\gamma,
\end{equation}
holds as long as the following bound holds:
\begin{equation}
    \label{est:logliptildeG}
    |H(\alpha,\beta) - H(\alpha',\beta)| \le C\eps_1 |\log \eps_1| \delta^\gamma.
\end{equation}
Combining \eqref{est:difftI1} and \eqref{est:difftI2}, we have thus proved
\begin{equation}
    \label{est:tI'holder}
    \|\tilde I'\|_{C^\gamma} \le C\eps_1^2 |\log\eps_1| \|h\|_{C^1}.
\end{equation}
The desired estimate \eqref{est:tGintegral} holds after combining \eqref{est:tILinfty}, \eqref{est:tI'Linfty}, and \eqref{est:tI'holder}.

Hence, we are left to show the estimate \eqref{est:logliptildeG}. We note that we can write $H(\alpha,\beta) = \tilde{H}(F(\alpha,\beta))$, where $\tilde{H}(z) = \tilde{G}'(|z|)\frac{z}{|z|}$. We first show that
\begin{equation}
    \label{est:logliptildeG1}
    |\tilde{H}(z) - \tilde{H}(w)| \le C\omega(|z-w|),
\end{equation}
for all $z, w \in \mathbb{C}$ such that $|z|, |w| < 1$. Here, $\omega$ denotes the Log-Lipschitz modulus of continuity. Without loss of generality we assume that $|w| \le |z|$. Note that
\begin{equation}\label{modulustH}
    \tilde{H}(z) - \tilde{H}(w) = \left(\tilde{G}'(|z|) - \tilde{G}'(|w|)\right)\frac{z}{|z|} + \tilde{G}'(|w|)\left(\frac{z}{|z|}-\frac{w}{|w|}\right).
\end{equation}
Since $\tilde{G}'$ is Log-Lipschitz, it is clear that the first term in \eqref{modulustH} is  bounded by $C\omega(|z-w|)$. To bound the second term in \eqref{modulustH}, we split into two cases. If $|w| \le 2|z-w|$, we have
\begin{align*}
    \left|\tilde{G}'(|w|)\left(\frac{z}{|z|}-\frac{w}{|w|}\right)\right| &\le 2|\tilde{G}'(|w|)| \le 2\omega(|w|) \le 2\omega(2|z-w|) \le C\omega(|z-w|),
\end{align*}
where we used that $\tilde{G}'(0) = 0$. If $|w| > 2|z-w|$, we note that
$$
|z-w|^2 = (|z| - |w|)^2 + |z||w| \left|\frac{z}{|z|}-\frac{w}{|w|}\right|^2 \ge |w|^2\left|\frac{z}{|z|}-\frac{w}{|w|}\right|^2,
$$
where we used $|w| \le |z|$. Using the above bound, we see that
\begin{align*}
    \left|\tilde{G}'(|w|)\left(\frac{z}{|z|}-\frac{w}{|w|}\right)\right| &\le C|w| |\log w| \frac{|z-w|}{|w|} \le C|\log (2|z-w|)| |z-w| \le C\omega(|z-w|),
\end{align*}
which concludes the proof of \eqref{est:logliptildeG1}. Finally to prove \eqref{est:logliptildeG}, we note that
\begin{align*}
    |H(\alpha,\beta) - H(\alpha',\beta)| &\le C\omega(|F(\alpha,\beta) - F(\alpha',\beta)|)\\
    &\le C\omega(|f_1(\alpha) - f_1(\alpha') + g_1(\alpha) - g_1(\alpha')|)\\
    &\le C\omega((\|f_1\|_{C^1} + \|g_1\|_{C^1})\delta)\\
    &\le C\eps_1\delta |\log(\eps_1\delta)| \le C\eps_1|\log \eps_1| \delta^\gamma,
\end{align*}
where we used $\gamma \in (0,1)$ in the final inequality above. 
\end{proof}
With the above three technical lemmas established, we demonstrate that the dynamics described by $Z_\pm$ is indeed dominant in the small-scale regime, namely $R \ll 1$. More precisely, we show the following stability result concerning the perturbation $\ze_\pm$:

\begin{prop}\label{prop:pertest}
    Let $\gamma \in (0,1)$ and $\ze_\pm$ be the unique $C^{1,\gamma}$ solution to \eqref{contourper} equipped with zero initial data. For any $\delta_0 \in (0,1)$ sufficiently small, there exists $R_0(\delta_0) > 0$ such that for any $R \in (0,R_0]$, the time instance
    $$
    \tau(\delta_0) = \sup\{t \ge 0\;|\; \max_\pm \sup_{0\le s \le t}(\|\ze_\pm(\cdot,s)\|_{C^{1,\gamma}}) \le \delta_0\}
    $$
    verifies the bound
    $$
    \tau(\delta_0) > T_0 = \frac{9\pi}{4}.
    $$
\end{prop}

\begin{proof}
    We prove by an energy argument. To start with, we choose $\delta_0 > 0$ such that $\delta_0 < \eps_0$, where $\eps_0$ is given in Lemma \ref{lem:logintegral}. We first note that $\tau > 0$ due to local well-posedness and $\ze_\pm(\alpha, 0) = 0$. Recall that $Z_\pm$ verifies the following bounds:
    $$
    \|Z_\pm\|_{C^{1,\gamma}} \le C_\pm R,
    $$
    where $C_\pm > 0$ are universal constants. Then we choose $R_0 > 0$ so that $C_\pm R_0 \le \min\{\eps_0, \eps_1, \eps_2\}$, where $\eps_i$, $i = 0,1,2$, are defined in Lemma \ref{lem:logintegral}, \ref{lem:logintegral2}, and \ref{lem:tGintegral}. Hence, invoking Lemma \ref{lem:logintegral}, Lemma \ref{lem:tGintegral}, and that the Hilbert transform $\calH$ is a bounded operator on $C^\gamma$, we obtain the following inequality for any $t \in [0,\tau)$, $R \in (0,R_0]$:
    \begin{equation}
        \label{est:ze-apriori}
        \begin{aligned}
        \frac{d}{dt}\|\ze_-(\cdot,t)\|_{C^{1,\gamma}} &\le C(\gamma)\left(\|\ze_-\|_{C^{1,\gamma}} + \|\ze_-\|_{C^{1,\gamma}}^2 + (1+\delta_0)R^2|\log R|\right)\\
        &\le C(\gamma)(1+\delta_0)(\|\ze_-\|_{C^{1,\gamma}} + R^2|\log R|).
        \end{aligned}
    \end{equation}
    To show an a priori estimate for $\|\ze_+\|_{C^{1,\gamma}}$, we have all the ingredients except for an $C^{1,\gamma}$ estimate for the integral
    $$
    \int_0^{2\pi}\log\left|3e^{-i\frac49 t} - e^{-i(\alpha + \beta)}\right|\pb \ze_+(\beta) d\beta.
    $$
    However, note that $3e^{-i\frac49 t} - e^{-i(\alpha + \beta)}$ is smooth in both $\alpha$ and $\beta$, and that the following lower bound holds:
    $$
    |3e^{-i\frac49 t} - e^{-i(\alpha + \beta)}| \ge |3e^{-i\frac49 t}| - |e^{-i(\alpha + \beta)}| = 2.
    $$
    The above two facts imply that the $C^{1,\gamma}$ norm of the above integral is bounded by $\|\ze_+\|_{C^{1,\gamma}}$. This fact combined with Lemma \ref{lem:logintegral2} and \ref{lem:tGintegral} yields:
    \begin{equation}
        \label{est:ze+apriori}
        \begin{aligned}
        \frac{d}{dt}\|\ze_+(\cdot,t)\|_{C^{1,\gamma}} &\le C(\gamma)\left(\|\ze_+\|_{C^{1,\gamma}} + \|\ze_+\|_{C^{1,\gamma}}^2 + (1+\delta_0)R^2|\log R|\right)\\
        &\le C(\gamma)(1+\delta_0)(\|\ze_+\|_{C^{1,\gamma}} + R^2|\log R|).
        \end{aligned}
    \end{equation}
    Using $\ze_\pm(\alpha,0) = 0$ and applying Gr\"onwall inequality to \eqref{est:ze-apriori}, \eqref{est:ze+apriori}, we obtain
    $$
    \|\ze_\pm(\cdot,t)\|_{C^{1,\gamma}} \le \frac{1}{C(\gamma)(1+\delta_0)}e^{C(\gamma)(1+\delta_0)t} R^2|\log R| < \delta_0
    $$
    for all $t > 0$ such that
    $$
    t < \frac{1}{C(\gamma)(1+\delta_0)}\log\left(C(\gamma) \frac{\delta_0(1+\delta_0)}{R^2|\log R|}\right).
    $$
    Now, we further choose $R$ even smaller depending on $\delta_0$, such that the RHS of the above inequality is strictly larger than $\frac{9\pi}{4}$. We have thus completed the proof upon invoking the definition of $\tau$.
\end{proof}
With Proposition \ref{prop:pertest}, we have the following theorem concerning the merger of patches and the topology change for the canonical momentum $F$ in the screened system.
\begin{thm}
    For any $\gamma \in (0,1)$, there exist $\delta_0 > 0$ sufficiently small and $R_0(\delta_0) > 0$ such 
    that, for any $R \in (0,R_0]$, and $d\in \left(\frac{R}{4}, \frac{3R}{4}\right)$, for the 
    $C^{1,\gamma}$ patch solutions $\sigma_\pm(x,t) = \mathbf{1}_{\Omega_t^\pm}(x)$ with initial data $\tau_\pm(x) = \mathbf{1}_{\Omega_0^\pm}(x)$, where
    $$
    \Omega_0^+ = \{(x,y) \in \R^2\; |\; \frac{(x-d-2R)^2}{R^2} + \frac{y^2}{(2R)^2} < 1\},\quad \Omega_0^- = \{(x,y) \in \R^2\;|\; \frac{x^2 + y^2}{R^2} < 1\}, 
    $$
    there exist time instances $T_1, T_2, T_3, T_4 > 0$ with
$$
0 < T_1 < T_2 < T_3 <  T_4 < \frac{9\pi}{4},
$$
such that $\Omega_t^+ \cap \Omega_t^- \neq \emptyset$ for $t \in (T_2, T_3)$, and $\Omega_t^+ \cap \Omega_t^- = \emptyset$ for $t \in [0,T_1) \cup (T_4,\frac{9\pi}{4}]$. Moreover, the canonical momentum $F$ sees the topology change in the following sense: $\{F(\cdot,t) = 1\} \neq \emptyset$ for $t \in (T_2, T_3)$, and $\{F(\cdot,t) = 1\} = \emptyset$ for $t \in [0,T_1) \cup (T_4,\frac{9\pi}{4}]$.
\end{thm}
\begin{proof}
    The statement concerning the merger of $\Omega_t^\pm$ and existence of time instances $T_i$, $i = 1,\hdots,4$ follows from a synthesis of the discussion in Subsection \ref{subsect:lefthandedunscreen} and Proposition \ref{prop:pertest}. The statement concerning $F$ is straightforward in view of the identity \eqref{defn:Fpatch}.
\end{proof}

\begin{rem}
    While the above merger is realized for patch solutions, a similar result holds for smooth, compactly supported solutions. We sketch the procedure here: as in the patch case, we first find merger in the unscreened equations \eqref{reducedeq} with solution $\bar\sigma_\pm$. Let $\bar\sigma_+$ be any non-radial, uniformly rotating smooth solution to the 2D Euler equation supported in disk $B_R(d,0)$, $d \in \R$. The existence of such solutions has been shown in \cite{desing1,desing2}. Let us also fix $\max\{\bar \sigma_+\} = 1$. Then let $\bar\sigma_-$ to be a nonnegative, radial function supported in disk $B_R(0)$ with $\max\{\bar\sigma_-\} = 1$. Since $\bar\sigma_+$ is nonradial, by adjusting $d$ accordingly, we obtain  a time-periodic solution of \eqref{reducedeq}, such that $\Int(\supp(\bar\sigma_\pm))$ are disjoint initially, then overlap, and become disjoint again. Then merger for the smooth solutions of the screened equations \eqref{sigmaeq2} is obtained via energy estimates by choosing scale $R$ sufficiently small, in a similar spirit to the above analysis for patches. A first time merger follows also from Theorem \ref{stabmerg} for smooth solutions of the left-handed system. 
\end{rem}

\section{Reconnection for the right-handed equations}\la{recright}
\subsection{The unscreened equations}
We consider here the right-handed equations with the equation of state \eqref{pmright}. In this case, a direct algebraic manipulation gives:
\begin{align*}
    v_+ &= - \U(\sigma_-) + (\B(F) + \U(F)) = - \U(\sigma_-) + 2\nabla^\perp \mathbb K(F),\\
    v_- &=  \U(\sigma_+) - (\B(F) + \U(F)) =  \U(\sigma_+) - 2\nabla^\perp \mathbb K(F)
\end{align*}
The screened right-handed equations are thus equivalently written as
\begin{equation}\label{eq:fullsys}
\begin{cases}
    \pa_t \sigma_+ - \U(\sigma_-) \cdot \nabla \sigma_+ + 2\nabla^\perp \mathbb K(F)\cdot \nabla \sigma_+ = 0,\\
    \pa_t \sigma_- + \U(\sigma_+) \cdot \nabla \sigma_- - 2\nabla^\perp \mathbb K(F)\cdot \nabla \sigma_- = 0,
\end{cases}
\end{equation}
and the unscreened equations are:
\begin{equation}\label{eq:model2}
\begin{cases}
    \pa_t \sigma_+ - \U(\sigma_-) \cdot \nabla \sigma_+ = 0,\\
    \pa_t \sigma_- + \U(\sigma_+) \cdot \nabla \sigma_- = 0.
\end{cases}
\end{equation}
\subsection{A motivating example for the unscreened equations}
We start with a simple example in the case of point vortex dynamics to demonstrate a merger  for the unscreeened equations \eqref{eq:model2}. Consider the data for point vortices given by the following: for $(x_0, y_0)$ with $x_0 < 0$, $y_0 > 0$, let
    $$
    \sigma_+(0,x,y) = \delta(x-x_0, y-y_0) - \delta(x+x_0, y-y_0),
    $$
    $$
    \sigma_-(0,x,y) = \delta(x-x_0, y+y_0) - \delta(x+x_0, y+y_0).
    $$
    After plugging in \eqref{eq:model2}, the solution for $t > 0$ satisfies the following form:
    $$
    \sigma_\pm(t,x,y) = \delta(x-x(t), y\mp y(t)) - \delta(x+x(t), y\mp y(t)),
    $$
    while $(x(t),y(t))$ obeys the following ODE:
    \begin{align}
        \frac{dx}{dt} &= -\frac{1}{2\pi}\frac{-2y}{(2y)^2} - \frac{-1}{2\pi} \frac{-2y}{(2x)^2 + (2y)^2},\\
        \frac{dy}{dt} &= -\frac{-1}{2\pi}\frac{2x}{(2x)^2 + (2y)^2}.
    \end{align}
    Simplifying, we have
    \begin{align}
        \frac{dx}{dt} &= \frac{1}{4\pi}\left(\frac1y - \frac{y}{x^2 + y^2}\right) = \frac{1}{4\pi}\frac{x^2}{y(x^2 + y^2)},\\
        \frac{dy}{dt} &= \frac{1}{4\pi}\frac{x}{x^2 + y^2}.
    \end{align}
    Note that $x/y$ is a constant of motion:
    \begin{align*}
        \frac{d}{dt}\left(\frac{x}{y}\right) &= \frac{x'}{y} - \frac{xy'}{y^2}\\
        &= \frac{1}{4\pi}\left(\frac{x^2}{y^2(x^2 + y^2)} - \frac{x^2}{y^2(x^2 + y^2)}\right)\\
        &= 0,
    \end{align*}
    which implies $x(t) = \frac{x_0}{y_0}y(t)$. Plugging this equality into the equation for $x(t)$, we have
    $$
    \frac{dx}{dt} = \frac{1}{4\pi}\frac{1}{\frac{y_0}{x_0}(1+\frac{y_0^2}{x_0^2})}\frac1x \quad\Rightarrow \quad \frac{1}{2}\frac{d}{dt} (x^2) = \frac{1}{4\pi}\frac{1}{\frac{y_0}{x_0}(1+\frac{y_0^2}{x_0^2})} < 0,
    $$
    where the inequality follows from $x_0 < 0, y_0 > 0$. Denote 
    $$
    C_0 := \frac{1}{2\pi}\frac{1}{\frac{y_0}{x_0}(1+\frac{y_0^2}{x_0^2})},
    $$
    we see that $x(t)$, and hence $y(t)$ converges to $0$ in finite time, with the following explicit time for the merger of point vortices:
    $$
    T_* = -\frac{x_0^2}{C_0}.
    $$

\subsection{Merger for smooth solutions of the unscreened equations}
The point vortex solution that lead to a finite-time merger above obeys the following symmetry:
\begin{equation}\label{oddodd}
\begin{aligned}
    \sigma_\pm (t,x_1,x_2) = -\sigma_{\pm}(t,-x_1,x_2),\\
    \sigma_+(t,x_1,x_2) = \sigma_-(t,x_1,-x_2).
\end{aligned}
\end{equation}
In this section, we demonstrate the merger for a large class of smooth solutions equipped with the above symmetry, inspired by the point vortex configuration discussed in the previous section. We remark that the symmetry property \eqref{oddodd} is reminiscent of the well-known odd-odd symmetry in the context of incompressible Euler equations in two dimensions. The dynamics in our case is completely different. This is because $\sigma_\pm$ does not see the mirror images of themselves across the $x_2$ axis.

As a preliminary, we show that the two symmetries \eqref{oddodd} are conserved by smooth solutions to both the unscreened  \eqref{eq:model2} and the screened system \eqref{eq:fullsys}.

\begin{prop}
    \label{prop:symmetry}
    Let $\sigma_\pm$ be the smooth solution to \eqref{eq:model2} (or \eqref{eq:fullsys}) with initial datum $\tau_\pm$. If $\tau_\pm$ verifies the following properties:
    \begin{align*}
        \tau_\pm (x_1,x_2) = -\tau_{\pm}(-x_1,x_2),\\
    \tau_+(x_1,x_2) = \tau_-(x_1,-x_2),
    \end{align*}
    then \eqref{oddodd} holds for all $t > 0$.
\end{prop}
\begin{proof}
    We only prove the case for the screened system \eqref{eq:fullsys}, as the corresponding statement concerning \eqref{eq:model2} is similar and follows easily. We first show the preservation of $\sigma_+(t,x_1,x_2) = \sigma_-(t,x_1,-x_2)$. Define the reflection map
    $
    R(x_1,x_2) = (x_1,-x_2),
    $
    and denote 
    $$
    A := \nabla R = \begin{bmatrix}
        1 & 0\\
        0 & -1
    \end{bmatrix}.
    $$
    Also define $\eta_\mp (t,x) = \sigma_\pm(t,Rx)$. We first observe that for any smooth function $h$,
    $$
    (\Delta^{-1} h)(Rx) = \Delta^{-1} (h\circ R)(x),\quad ((I -\Delta)^{-1}h)(Rx) = (I -\Delta)^{-1}(h\circ R)(x).
    $$
    Writing $\nabla^\perp = J\nabla$, where
    $$
    J = \begin{bmatrix}
        0 & -1\\
        1 & 0
    \end{bmatrix},
    $$
    and using that $JA = -AJ$, we have
    \begin{align*}
        \U(h\circ R)(x) = J\nabla(\Delta^{-1} (h\circ R)(x)) = JA(\nabla\Delta^{-1} h(Rx)) = -A\nabla^\perp \Delta^{-1} h(Rx) = -A\U(h)(Rx).
    \end{align*}
    Similarly, we also have
    $$
    \B(h\circ R)(x) = -A\B(h) (Rx).
    $$
    Using the equation for $\sigma_-$ and the above computations, we have
    \begin{align*}
        \pt \eta_+(x) &= -\U(\sigma_+)(Rx) \cdot \nabla \sigma_-(Rx) + \frac12(\B(\sigma_+ + \sigma_-)(Rx) + \U(\sigma_+ + \sigma_-)(Rx)) \cdot \na \sigma_-(Rx)\\
        &= -\left(-A^{-1}\U(\eta_-)(x) \cdot A\na \eta_+(x)\right)\\
        &\quad + \frac12 \left(-A^{-1}\left(\B(\eta_+ + \eta_-)(x) + \U(\eta_+ + \eta_-)(x)\right)\right) \cdot A\na \eta_+(x)\\
        &= \U(\eta_-)(x) \cdot\na \eta_+(x) - \frac12 \left(\B(\eta_+ + \eta_-)(x) + \U(\eta_+ + \eta_-)(x)\right) \cdot \nabla \eta_+(x),
    \end{align*}
    where we used that $A$ is symmetric in the final equality above. A similar derivation using the $\sigma_+$ equation yields:
    $$
    \pt \eta_-(x) = -\U(\eta_+)(x) \cdot\na \eta_-(x) + \frac12 \left(\B(\eta_+ + \eta_-)(x) + \U(\eta_+ + \eta_-)(x)\right) \cdot \nabla \eta_-(x).
    $$
    Finally, observe that by the assumption on initial datum $\tau_\pm$,
    $$
    \eta_\pm (x) = \tau_\mp(Rx) = \tau_\pm (x).
    $$
    By uniqueness, we conclude that for any $t > 0$, $\sigma_\pm (t,x) = \eta_\pm(t,x) = \sigma_\mp(t, Rx) = \sigma_\mp (t, x_1,-x_2)$, which is the desired result. The conservation of odd-in-$x_1$ symmetry for $\sigma_\pm$ follows from a very similar argument. We omit further details.
\end{proof}

Another ingredient in the construction of a merger is the following key computation that concerns a restricted version of the center of mass. For convenience, let us denote $Q := (0,\infty)^2$. The following proposition holds:
\begin{prop}\label{prop:center}
    Let $\sigma_\pm(\cdot,t)$ be a smooth solution to \eqref{eq:model2} on $[0,T]$ for some $T > 0$. In addition, assume that $\sigma_\pm$ has the following properties:
    \begin{enumerate}
        \item The symmetry condition \eqref{oddodd} holds for $\sigma_\pm(x,0)$;
        \item $\{|\sigma_+| > 0\} \subset \R^2_+$ for all $t \in [0,T]$;
        \item $\sigma_+$ is either nonpositive or nonnegative in $Q$. 
    \end{enumerate}
    Define 
    $$
    E_j(t) = \int_Q x_j \sigma_+(x,t) dx,\quad j=1,2.
    $$
    Then the following identities hold:
    \begin{subequations}
    \begin{equation}
        \label{eq:E1'}
        E_1'(t) = \frac2\pi \int_{Q\times Q} \frac{x_1y_1(x_2 + y_2)}{|x-\bar y|^2 |x+y|^2}\sigma_+(y)\sigma_+(x)dydx > 0,
    \end{equation}  
    \begin{equation}
        \label{eq:E2'}
        E_2'(t) = \frac{1}{2\pi}\int_{Q\times Q} \frac{x_1 + y_1}{|x+y|^2}\sigma_+(y)\sigma_+(x)dydx > 0,
    \end{equation}  
    \end{subequations}
    where $\bar{y} = (y_1,-y_2)$.
\end{prop}
Note that if we additionally choose $\sigma_+$ to be nonpositive in $Q$, we have $E_j(t) < 0$ for all $t$ and they monotonically increase. This mirrors the point vortex scenario considered in the previous subsection. 
\begin{proof}
    Using the first equation of \eqref{eq:model2}, it holds that
    \begin{equation}
        \label{eq:Ej'aux1}
        \begin{aligned}
        E_j'(t) &= \int_Q x_j\U(\sigma_-)\cdot \nabla \sigma_+ dx= -\int_Q \U_j(\sigma_-)\sigma_+ dx + \int_{\p Q} x_j\sigma_+ (\U(\sigma_-)\cdot n)dS,
        \end{aligned}
    \end{equation}
    where $n$ denotes the outward-pointing normal vector along $\p Q$. Observe that this boundary term vanishes: if $j = 1$, only the boundary term along the positive $x_1$ axis survives. But by Property 2, $\sigma_+ \equiv 0$ along $x_1$ axis. If $j = 2$, then only the boundary term along the positive $x_2$ axis survives. But in this case, $\U(\sigma_-) \cdot n = -\U_1(\sigma_-)$ vanishes due to odd-in-$x_1$ symmetry. Hence, we have
    \begin{equation}
        \label{eq:Ej'aux2}
        E_j'(t) = -\int_Q \U_j(\sigma_-)\sigma_+ dx.
    \end{equation}
    We first consider the case $j = 1$. Note that
    \begin{align*}
        \U_1(\sigma_-)(x_1,x_2) &= -\frac{1}{2\pi}\int_{\R^2} \frac{x_2 - y_2}{|x-y|^2} \sigma_-(y) dy_2 dy_1\\
        &= -\frac{1}{2\pi}\int_0^\infty \int_{\R} \left(\frac{x_2 - y_2}{|x-y|^2} - \frac{x_2 - y_2}{|x-\tilde{y}|^2}\right) \sigma_-(y)dy_2 dy_1\\
        &= -\frac2\pi \int_0^\infty \int_{\R} \frac{x_1y_1(x_2 - y_2)}{|x-y|^2|x-\tilde{y}|^2} \sigma_-(y)dy_2 dy_1,
    \end{align*}
    where $\tilde{y} = (-y_1,y_2)$. Here we used the symmetry $\sigma_-(y_1,y_2)= -\sigma_-(-y_1,y_2)$ in the second equality above. Now, using the other symmetry $\sigma_-(y_1, y_2) = \sigma_+(y_1,-y_2)$ and Property 2:
    \begin{align*}
        \U_1(\sigma_-)(x_1,x_2) &= -\frac2\pi \int_0^\infty \int_{\R} \frac{x_1y_1(x_2 - y_2)}{|x-y|^2|x-\tilde{y}|^2} \sigma_+(y_1,-y_2)dy_2 dy_1\\
        &= -\frac2\pi \int_0^\infty \int_{-\infty}^0 \frac{x_1y_1(x_2 - y_2)}{|x-y|^2|x-\tilde{y}|^2} \sigma_+(y_1,-y_2)dy_2 dy_1\\
        &= -\frac2\pi \int_Q \frac{x_1y_1(x_2 + y_2)}{|x-\bar{y}|^2|x+y|^2} \sigma_+(y_1,y_2)dy_2 dy_1
    \end{align*}
    Combining the above identity with \eqref{eq:Ej'aux2} yields \eqref{eq:E1'}.\\
    For the case $j = 2$, we again use symmetry \eqref{oddodd} and Property 2 to obtain:
    \begin{align*}
        \U_2(\sigma_-)(x_1,x_2) &= \frac{1}{2\pi}\int_{\R^2}\frac{x_1-y_1}{|x-y|^2}\sigma_-(y_1,y_2) dy_2 dy_1\\
        &= \frac{1}{2\pi}\int_{\R}\int_{-\infty}^0\frac{x_1-y_1}{|x-y|^2}\sigma_+(y_1,-y_2) dy_2 dy_1\\
        &= \frac{1}{2\pi}\int_{\R}\int^{\infty}_0\frac{x_1-y_1}{|x-\bar y|^2}\sigma_+(y_1,y_2) dy_2 dy_1\\
        &= \frac{1}{2\pi}\bigg[\int_0^\infty\int^{\infty}_0\frac{x_1-y_1}{|x-\bar y|^2}\sigma_+(y_1,y_2) dy_2 dy_1 + \int_{-\infty}^0\int^{\infty}_0\frac{x_1-y_1}{|x-\bar y|^2}\sigma_+(y_1,y_2) dy_2 dy_1\bigg]\\
        &= \frac{1}{2\pi}\int_Q \left(\frac{x_1-y_1}{|x-\bar y|^2} - \frac{x_1+y_1}{|x+ y|^2}\right)\sigma_+(y_1,y_2) dy_2 dy_1\\
        &= \frac{1}{2\pi}\int_Q \frac{2y_1(x_1^2 - y_1^2 -(x_2+y_2)^2)}{|x-\bar y|^2 |x+y|^2}\sigma_+(y_1,y_2) dy_2 dy_1.
    \end{align*}
    Plugging into \eqref{eq:Ej'aux2}, we have
    \begin{equation}\label{E2'aux1}
        E_2'(t) = -\frac{1}{2\pi}\int_{Q\times Q} \frac{2y_1(x_1^2 - y_1^2 -(x_2+y_2)^2)}{|x-\bar y|^2 |x+y|^2}\sigma_+(y)\sigma_+(x) dy dx.
    \end{equation}
    Note that
    \begin{align*}
        \frac{2y_1(x_1^2 - y_1^2 -(x_2+y_2)^2)}{|x-\bar y|^2 |x+y|^2} + \frac{2x_1(y_1^2 - x_1^2 -(x_2+y_2)^2)}{|x-\bar y|^2 |x+y|^2} = -2\frac{x_1 + y_1}{|x+y|^2}.
    \end{align*}
    Then after a symmetrization, from \eqref{E2'aux1} we immediately have
    \begin{align*}
        E_2'(t) = \frac{1}{2\pi}\int_{Q\times Q}\frac{x_1 + y_1}{|x+y|^2}\sigma_+(y)\sigma_+(x) dy dx,
    \end{align*}
    which is \eqref{eq:E2'}.
\end{proof}
With the above Proposition, we show that a family of initially disjoint $\sigma_\pm$ data can merge in finite time:
\begin{thm}\label{thm:overlap}
    There exists smooth, compactly supported solutions $\sigma_\pm$ to \eqref{eq:model2} such that\\ $\supp(\sigma_+(\cdot,0)) \cap \supp(\sigma_-(\cdot,0)) = \emptyset$, and $\Int(\supp(\sigma_+(\cdot,T))) \cap \Int(\supp(\sigma_-(\cdot,T))) \neq \emptyset$ for some $T > 0$.
\end{thm}
\begin{proof}
    Choose initial data $\tau_\pm(x) = \sigma_\pm(x,0)$ such that the following properties are satisfied:
    \begin{enumerate}
        \item $\tau_\pm \le 0$ in $Q$, $\{\sigma_+ < 0\} \subset Q$ is simply connected, and the symmetry condition \eqref{oddodd} holds;
        \item $\|\tau_+\|_{L^\infty} = 1$, $|\supp(\tau_+)\cap Q| = 2$, $|E_j(0)| > \frac12$ for $j = 1,2$;
        \item Let $D_t = \{x \in \supp(\sigma_+(x,t))\cap Q\;|\; \sigma_+(x,t) \le -\frac12\}$. Choose initial data such that 
            \begin{equation}\la{epschoice}
            |(\supp(\tau_+)\cap Q)\backslash D_0| = \eps > 0,
            \end{equation}
    where $\eps < 1$ is a sufficiently small number to be chosen later;
        \item For all $\alpha \in (-1, 0)$, $\alpha$ is a regular value of $\tau_+|_Q$. Moreover, $\{\tau_+|_Q = -1\}$ contains only one point, which is a nondegenerate critical point.
    \end{enumerate}
    We remark that all properties listed above are only chosen out of convenience. We expect that more flexible initial data would work.
    
    Define 
    $$
    T_0 := \sup\{\tau > 0\;|\; \{|\sigma_+(\cdot,t)| > 0\} \subset \R_+^2 \text{ for all } t\in [0,\tau)\}.
    $$
    Note that $T_0$ is well-defined since, by local well-posedness, there exists $\tau_0 > 0$ such that $\{|\sigma_+(\cdot,t)| > 0\} \subset \R_+^2 \text{ for all } t\in [0,\tau_0)$. Moreover, $|D_t| = |D_0|$, $|(\supp(\sigma_+(x,t))\cap Q)\backslash D_t| = \eps$ for all $t \in [0,T_0)$. By the choice of initial data and Proposition \ref{prop:center}, we note that for $j = 1,2$, $t \in [0,T_0)$:
    \begin{equation}\label{est:monotone}
    |E_j(t)| = -E_j(t) \le -E_j(0) = |E_j(0)|
    \end{equation}
    Due to the monotone quantities \eqref{est:monotone}, we first show that $E$ is confined: by Chebyshev's inequality,
    \begin{equation}\label{est:x1constraint}
    \begin{aligned}
    |D_t \cap \{x_1 \ge 4|E_1(0)|\}| &= |D_t \cap \{x_1(-\sigma_+(x,t)) \ge 4|E_1(0)|\}| \\
    &\le \frac{1}{4|E_1(0)|}\int_{Q} x_1(-\sigma_+(x,t)) dx \le \frac14.
    \end{aligned}
    \end{equation}
    Similarly, we also have $|D_t \cap \{x_2 \ge 4|E_2(0)|\}| \le \frac14$. Define $Q' = [0,4|E_1(0)|] \times [0,4|E_2(0)|] \supset [0,2]^2$. Then using $|D_t| = |D_0| = 2-\eps$,
    $$
    |D_t \cap Q'| \ge 2-\eps - \frac12\ge 1
    $$
    for all $\eps < \frac12$. Moreover, consider the strip $S = [0,\min(\frac{1}{8|E_2(0)| },2|E_1(0)|)] \times [0, 4|E_2(0)|]$ and define $D_t' = (D_t \cap Q') \backslash S$. It is clear that $|D_t'| \ge 1 - \frac12 = \frac12$. Moreover, for all $x, y \in D_t'$, it holds that
    $$
    x_1 + y_1 \in \left[\min(\frac{1}{4|E_2(0)|}, 4|E_1(0)|), 8|E_1(0)|\right],
    $$
    $$
    x_2 + y_2 \in [0,8|E_2(0)|].
    $$
    From which we get
    $$
    \frac{x_1 + y_1}{|x+y|^2} \ge \frac{1}{64}\frac{\min(\frac{1}{4|E_2(0)|}, 4|E_1(0)|)}{|E_1(0)|^2 + |E_2(0)|^2},\quad x,y \in D_t'.
    $$
    Since the kernel $\frac{x_1 + y_1}{|x+y|^2}$ is nonnegative in $Q\times Q$, we conclude from \eqref{eq:E2'} and the definition of $D_t$ that
    \begin{equation}\label{est:E2coercive}
    E_2'(t) \ge \frac{1}{8\pi}\int_{D_t' \times D_t'} \frac{1}{64}\frac{\min(\frac{1}{4|E_2(0)|}, 4|E_1(0)|)}{|E_1(0)|^2 + |E_2(0)|^2} dx \ge C\frac{\min(\frac{1}{4|E_2(0)|}, 4|E_1(0)|)}{|E_1(0)|^2 + |E_2(0)|^2} =: c_2> 0,
    \end{equation}
    But this immediately implies that $T_0 < \infty$: otherwise there exists $T_1 = \frac{|E_2(0)|}{c_2} < \infty$ such that $E_2(T_1) > 0$, which is absurd.

    To prove the main result, we need several topological arguments. By the symmetry assumption \eqref{oddodd}, it suffices to consider the dynamics in the right-half plane. Define open sets
    $$
    \calD_t^\pm := \{x \in \{x_1 > 0\}\;|\; |\sigma_\pm(x,t)| > 0\}.
    $$
    Since $\calD_0^\pm$ is simply connected by construction, so is $\calD_t^\pm$ because these sets are transported by smooth vector fields. We start with showing the following two claims:\\
    
    \noindent\textbf{Claim 1.} There exists $\delta_1 > 0$ such that $\calD_T^+ \cap \R_+^2 \neq \emptyset$ for all $T \in [T_0, T_0 + \delta_1)$;\\
    \noindent\textbf{Claim 2.} For any $\delta > 0$, there exists $T \in (T_0, T_0 + \delta)$ such that $\calD_T^+ \cap \{x_2 \le 0\} \neq \emptyset$.\\
    
    To prove Claim 1, we first make the following observation: for any $t \in [0,T_0)$, there exists $x(t) \in \calD_t^+$ such that $x_2(t) \ge \frac{3}{16|E_1(0)|} > 0$: suppose not, then there exists $t < T_0$ such that $D_t \subset \calD_t^+ \subset \{0<x_2 < \frac{3}{16|E_1(0)|}\}$. Hence
    $$
    |D_t \cap \{x_1 < 4|E_1(0)|\}| < \frac{3}{16|E_1(0)|}\cdot 4|E_1(0)| = \frac34.
    $$
    However this above inequality contradicts \eqref{est:x1constraint}. Then this observation together with a continuity argument immediately implies Claim 1.
    

    To prove Claim 2, we see that if the claim were false, then there exists $\delta > 0$ such that for all $T \in (T_0, T_0 + \delta)$, it holds that $\calD_T^+ \subset \{x_2 > 0\}$. This contradicts the definition of $T_0$ as a supremum.

    Now choose $\delta_1 > 0$ as in Claim 1. The above two claims inform us that there exists $T \in (T_0, T_0 + \delta_1)$ such that $\calD_T^+ \cap \{x_2 > 0\}$ and $\calD_T^+ \cap \{x_2 \le 0\}$ are both nonempty. Since $\calD_T^+$ is open, we further have
    $$
    \calD_T^+ \cap \{x_2 > 0\}\neq \emptyset,\quad \calD_T^+ \cap \{x_2 < 0\} \neq \emptyset.
    $$
    Since $\calD_T^+ \subset \R^2$ is connected and thus path-connected, there exists a continuous path $\gamma:[0,1] \to \calD_T^+$ such that $\gamma(0) \in \{x_2 > 0\}$ and $\gamma(1) \in \{x_2 < 0\}$. By Intermediate Value Theorem, there exists $\alpha \in (0,1)$ such that $\gamma(\alpha) \in \{x_2 = 0\}$. However by the symmetry $\sigma_+ (x_1,x_2) = \sigma_-(x_1,-x_2)$, we must also have $\gamma(\alpha) \in \calD_T^-$. This shows that $\calD_T^+ \cap \calD_T^- \neq \emptyset$ and the Theorem is proved.
\end{proof}

As a corollary of Theorem \ref{thm:overlap}, the support of the canonical momentum variable $F = \frac{\sigma_+ + \sigma_-}{2}$ sees topological change in the following sense:
\begin{cor}\label{cor:topo}
    There exists smooth, compactly supported solutions $\sigma_\pm$ to \eqref{eq:model2} such that $F = \frac{\sigma_+ + \sigma_-}{2}$ has the following properties: there exists $T > 0$ such that
    \begin{enumerate}
        \item $\Int(\supp(F(\cdot,0)))$ has 4 connected components, but $\Int(\supp(F(\cdot,T)))$ has 2 connected components;
        \item the following trichotomy holds:
        \begin{enumerate}
            \item $\{|F(x,0)| > \frac12\} = \emptyset$, but $\{|F(x,T)| > \frac12\} \neq \emptyset$, or;
            \item $Card(\{|F(x,0)| = \frac12\}) = 4$, but $Card(\{|F(x,T)| = \frac12\}) > 4$, or;
            \item There exists $|\alpha| \in (0,\frac12)$ such that $\alpha$ is a regular value of $F(x,0)$, but a critical value of $F(x,T)$.
        \end{enumerate}
    \end{enumerate}
\end{cor}
\begin{proof}
Let the initial data $\tau_\pm$ and the time instance $T > 0$ be specified as in Theorem \ref{thm:overlap}. We first remark that one can restrict the attention to the dynamics in the right-half plane $\{x_1 > 0\}$. Recalling the sets $\calD_t^\pm$, we note that
$$
\Int(\supp(\sigma_\pm(\cdot,t))) = \calD_t^\pm \cup \tilde{\calD}_t^\pm,
$$
where we denote $\tilde{A} := \{(-x_1, x_2) \in A\}$ for any set $A$. By choice of initial data, we have $\calD_0^\pm \cap \tilde{\calD}_0^\pm = \emptyset$. Then by odd-in-$x_1$ symmetry recorded in \eqref{oddodd}, we must have $\calD_t^\pm \cap \tilde{\calD}_t^\pm = \emptyset$ for all $t \ge 0$. Therefore, in view of $\Int(\supp(F(\cdot,t))) \subset \Int(\supp(\sigma_+(\cdot,t))) \cup \Int(\supp(\sigma_-(\cdot,t)))$, it suffices to understand $S_t^> := \{x \in \{x_1 > 0\}\;|\;F(x,t) \neq 0\}$.
    \begin{enumerate}
        \item To prove the first statement, we first show that $S_t^> = \calD_t^+ \cup \calD_t^-$. $S_t^> \subset \calD_t^+ \cup \calD_t^-$ is clear by definition of $F$. To show the other direction, observe that for $x_1 > 0$,
        $$
        F(x,t) = 
        \begin{cases}
            \frac{\sigma_+(x,t)}{2},& x\in \calD_t^+ \backslash \calD_t^-,\\
            \frac{\sigma_-(x,t)}{2},& x\in \calD_t^- \backslash \calD_t^+,\\
            \frac{\sigma_+(x,t) + \sigma_-(x,t)}{2},& x\in \calD_t^+ \cap \calD_t^-,\\
            0,& \text{elsewhere}
        \end{cases}
        $$
        It is clear that $F \neq 0$ whenever $x \in \calD_t^+ \Delta \calD_t^-$. By \eqref{oddodd}, we see that $\frac{\sigma_+(x,t) + \sigma_-(x,t)}{2} \neq 0$ for all $x\in \calD_t^+ \cap \calD_t^-$. Hence, we arrive at $\calD_t^+ \cap \calD_t^- \subset S_t^>$.

        By choice of initial data, it is immediate that $S_0^>$ has 2 connected components since $\calD_0^\pm$ are disjoint, connected open sets. By Theorem \ref{thm:overlap}, we have $\calD_T^+ \cap \calD_T^- \neq \emptyset$. This already implies that $S_t^>= \calD_t^+ \cup \calD_t^-$ is a connected open set.
        \item From the construction of initial data, it is straightforward to see that $\{|F(x,0)| > \frac12\} = \emptyset$ and $\alpha$ is a regular value of $F(x,0)$ for any $\alpha \in (-\frac12,\frac12)\backslash\{0\}$. Moreover, the set $\{F(x,0) = \frac12\} = \{|\tau_\pm| = 1\}$ contains exactly 4 distinct points. 
        
        At time instance $T$, we show the following claim:\\
        
        \noindent\textbf{Claim:} There exists $x^* =(x_1^*,0) \in \calD_T^+ \cap \calD_T^-$ such that $F(x^*,T) < 0$ and $\nabla F(x^*,T) = 0$.\\
        
        Before we prove this claim, we explain how one arrives at the trichotomy from this result: if there exists some $x \in \calD_T^+ \cap \calD_T^-$ such that $F(x,T) < -\frac12$, this trivially falls into the first scenario. If $F(x,T) \ge -\frac12$ for all $x \in \calD_T^+ \cap \calD_T^-$, we further split into two cases. Let $x^*$ be specified by the claim. If $F(x^*,T) > -\frac12$, then this falls into the third scenario. Then it suffices to consider $F(x^*,T) = -\frac12$. Since $F(x,T) \ge -\frac12$ and $\sigma_\pm(x,T) < 0$ for all $x \in \calD_T^+ \cap \calD_T^-$, then $\sigma_\pm(x,T) > -1$. Therefore, $\{\sigma_\pm(x,T) = -1\} \subset S_T^>\backslash(\calD_T^+ \cap \calD_T^-)$. Moreover, since diffeomorphisms conserve critical points, $Card(\{|\sigma_\pm(x,T)| = 1\}) = 4$. From the above two facts and the symmetry assumption \eqref{oddodd}, we conclude that $Card(\{|F(x,T)| = \frac12\}) \ge 4$. However, because $F(x^*, T) = -\frac12$ and $x^* \in \calD_T^+ \cap \calD_T^-$, we actually have $Card(\{|F(x,T)| = \frac12\}) > 4$, which falls into the second scenario.

        Now it remains to prove the claim. Consider the set $E = \calD_T^+ \cap \calD_T^- \cap \{x_2 = 0\}$, which is clearly nonempty and open in the induced topology of $\{x_2 = 0\}$. Let us fix a connected component $C$ of $E \subset \{x_2 = 0\}$. Then there exist $0 < a < b < \infty$ such that $C = (a,b) \times \{0\}$. Note that $\sigma_+((a,0),T) = \sigma_+((b,0),T) = 0$ and $\sigma_+((x_1,0),T) < 0$ for any $x_1 \in (a,b)$. By Mean Value Theorem, there exists $x_1^* \in (a,b)$ such that $\p_1 \sigma_+((x_1^*,0),T) = 0.$ On the other hand, we note that for any $(z,0) \in \R^2$, $t \ge 0$,
        $$
        \p_1 F(z,0) = \frac12(\p_1\sigma_+(z,0) + \p_1\sigma_-(z,0)) = \frac12(\p_1\sigma_+(z,0) + \p_1\sigma_+(z,0)) = \p_1\sigma_+(z,0),
        $$
        $$
        \p_2 F(z,0) = \frac12(\p_2\sigma_+(z,0) + \p_2\sigma_-(z,0)) = \frac12(\p_2\sigma_+(z,0) - \p_2\sigma_+(z,0)) = 0,
        $$
        where we used the symmetry property \eqref{oddodd} above. Therefore, we conclude that there exists $x_1^* \in (a,b)$ such that $\nabla F((x_1^*,0),T) = 0$. Moreover as $\sigma_+((x_1^*,0), T) < 0$, we also have $\sigma_-((x_1^*,0), T) =\sigma_+((x_1^*,0), T) < 0$ by symmetry, and hence $F((x_1^*,0), T) < 0$. This concludes the proof of the claim.
    \end{enumerate}
\end{proof}
\subsection{Merger for smooth solutions of the screened equations}
In view of the  above analysis for the unscreened equations \eqref{eq:model2}, we prove an analogous result to Theorem \ref{thm:overlap} and Corollary \ref{cor:topo} for the screened  system \eqref{eq:fullsys} when the dynamics are at sufficiently small scales. Toward this end, given a scale $R > 0$, we introduce the following change of variables:
\begin{equation}
    \label{defn:rescale}
    \sigma_\pm (x,t) = \tilde\sigma_\pm (z,t),\quad z = \frac{x}{R}.
\end{equation}
From \eqref{eq:fullsys}, we obtain the following equivalent system satisfied by profiles $\tilde\sigma_\pm$ in $(z,t)$ coordinates:
\begin{equation}
    \label{eq:tildefullsys}
    \begin{cases}
    \pa_t \tsig_+ - \U(\tsig_-) \cdot \nabla \tsig_+ + \frac{R}{2}\left[\int_{\R^2}\calK(R(z-w))(\tsig_+(w) + \tsig_-(w))dw\right]\cdot \nabla \tsig_+ = 0,\\
    \pa_t \tsig_- + \U(\tsig_+) \cdot \nabla \tsig_- - \frac{R}{2}\left[\int_{\R^2}\calK(R(z-w))(\tsig_+(w) + \tsig_-(w))dw\right]\cdot \nabla \tsig_- = 0,
\end{cases}
\end{equation}
where $\calK$ is the kernel corresponding to operator $2\nabla^\perp \mathbb K$.

As a preliminary, we remark on the symmetry property concerning $\tsig_\pm$ and several a priori bounds. First, we recall from Proposition \ref{prop:symmetry} that the solution $\sigma_\pm$ to the full system \eqref{eq:fullsys} conserves the symmetry \eqref{oddodd}. Then so does $\tsig_\pm$. Moreover, the following a priori bounds hold: for any $t \ge 0$
\begin{equation}
    \label{est:tildeapriori}
    \begin{aligned}
        \|\tsig_\pm(\cdot,t)\|_{L^\infty} = \|\tsig_\pm(\cdot,t)\|_{L^\infty} = \|\tau_\pm\|_{L^\infty},\\
        \|\tsig_\pm(\cdot,t)\|_{L^1} = R^{-2}\|\tsig_\pm(\cdot,t)\|_{L^1} = R^{-2}\|\tau_\pm\|_{L^1}.
    \end{aligned}
\end{equation}

We now set the initial data $\tau_\pm$ to the original system \eqref{eq:fullsys} by
\begin{equation}
    \label{eq:init}
    \tau_\pm(x) = \tilde\tau_\pm(z), 
\end{equation}
where $\tilde\tau_\pm$ satisfies Properties 1--4 listed in the proof of Theorem \ref{thm:overlap}. Let $\bar\sigma_\pm(z,t)$ be the solution to the leading order problem \eqref{eq:model2} initiated by data $\tilde\tau_\pm(z)$. Let us also denote
$$
\bar E_j (t) = \int_Q z_j \bar\sigma_+(z)dz.
$$

By Theorem \ref{thm:overlap}, there exists $\bar T:= \frac{|\bar E_2(0)|}{c_2} < \infty$ such that there is $T \in (0,\bar T)$ such that the support interior of $\bar\sigma_\pm$ intersects at time $T$. Here, we recall that $c_2$ defined in \eqref{est:E2coercive} only depends on $\bar E_j(0)$, $j = 1,2$. In particular, $c_2$ is independent of the scale $R$. 

With above important notations introduced, we state and prove the following main theorem concerning the merger for the original screened equations \eqref{eq:fullsys}:
\begin{thm}\label{thm:overlapfull}
    Let $\sigma_\pm$ be the solution to \eqref{eq:fullsys} initiated by data given in \eqref{eq:init}. Then $\supp(\sigma_+(\cdot,0)) \cap \supp(\sigma_-(\cdot,0)) = \emptyset$, and $\Int(\supp(\sigma_+(\cdot,T))) \cap \Int(\supp(\sigma_-(\cdot,T))) \neq \emptyset$ for some $T > 0$.
\end{thm}
\begin{proof}
    The proof is arranged into various steps:
    
    \noindent\textbf{Step 1. Control of the supports for $\sigma_\pm$.} In this step, we establish an estimate for the size of the supports for $\sigma_\pm$ (and therefore $\tsig_\pm$) on a sufficiently long time interval, namely on $[0, 2\bar T]$. Let $d_0 >0$ be such that $\supp(\tilde\tau_\pm) \subset B_{d_0}(0)$ i.e. $\supp(\tau_\pm) \subset B_{d_0R}(0)$.
    Invoking a standard integral operator estimate, there exists a universal constant $C_0 > 0$ such that
    $$
    \|\mathbb{U}(\sigma_\pm)\|_{L^\infty} + \|(\mathbb{B} + \mathbb{U})(F)\|_{L^\infty} \le C_0\|\sigma_\pm\|_{L^1}^\frac12 \|\sigma_\pm\|_{L^\infty}^\frac12 \le C_0 R\|\tilde\tau_\pm\|_{L^1}^\frac12.
    $$
    Here, we used \eqref{est:tildeapriori} and $\|\tilde\tau_\pm\|_{L^\infty} = 1$. The above estimate together with \eqref{defn:rescale} immediately gives
    \begin{equation}
        \label{est:suppbd}
        \supp(\tsig_\pm(\cdot,t)) \subset B_{d_{\bar T}}(0),\quad t \in [0,2\bar T],
    \end{equation}
    where $d_{\bar T} = d_0 + 2C_0 \bar T \|\tilde\tau_\pm\|_{L^1}^\frac12$. In particular, $d_{\bar T}$ is independent of $R$.\\

    \noindent\textbf{Step 2. Moment analysis.} In this step, similar to the study of the leading order model \eqref{eq:model2}, we analyze the restricted moments corresponding to the profiles $\tsig_+$, namely:
    \begin{equation}
        \label{defn:tildemoment}
        E_j(t) = \int_Q z_j \tilde\sigma_+(z) dz.
    \end{equation}
    Note that we have $E_j(0) = \bar E_j(0)$, and $E_j(t) \le 0$ for all $t \ge 0$ due to our choice of initial data.
    
    The main ingredient is the following lemma, analogous to Proposition \ref{prop:center}:
    \begin{lemma}
        Let $\sigma_\pm(\cdot,t)$ be a smooth solution to \eqref{eq:fullsys} on $[0,T]$, $T > 0$. In addition, assume that $\sigma_\pm$ has the following properties:
    \begin{enumerate}
        \item The symmetry condition \eqref{oddodd} holds for $\sigma_\pm(\cdot,0)$;
        \item $\{|\sigma_+| > 0\} \subset \R^2_+$ for all $t \in [0,T]$;
        \item $\sigma_+$ is either nonpositive or nonnegative in $Q$. 
    \end{enumerate}
    Then the following identities hold:
    \begin{subequations}
    \begin{equation}
        \label{eq:tE1'}
        E_1'(t) = I_1[\tsig_+] + RJ_1[\tilde\sigma]  ,
    \end{equation}  
    \begin{equation}
        \label{eq:tE2'}
        E_2'(t) = I_2[\tsig_+] - RJ_2[\tilde\sigma],
    \end{equation}  
    \end{subequations}
    where
    \begin{align*}
        I_1[\tsig_+] &:= \frac2\pi \int_{Q\times Q} \frac{z_1w_1(z_2 + w_2)}{|z-\bar w|^2 |z+w|^2}\tsig_+(w)\tsig_+(z)dwdz,\\
        I_2[\tsig_+] &:= \frac{1}{2\pi}\int_{Q\times Q} \frac{z_1 + w_1}{|z+w|^2}\tsig_+(w)\tsig_+(z)dwdz,\\
        J_i[\tilde\sigma] &:= \frac{1}{2}\int_Q\int_{\R^2} \calK_i(R(z-w))\left(\tsig_+(w) + \tsig_-(w)\right)\tsig_+(z)dwdz,\quad i = 1,2.
    \end{align*}
    \end{lemma}
    The proof for this lemma is almost identical to that of Proposition \ref{prop:center}, where one uses \eqref{eq:tildefullsys} instead of \eqref{eq:model2}. We therefore omit the tedious computations here.

    To proceed, we define 
    $$
    \overline{T}_0 := \sup\{\tau > 0\;|\; \{|\tilde\sigma_+(\cdot,t)| > 0\} \subset \R_+^2 \text{ for all } t\in [0,\tau)\}.
    $$
    For the sake of contradiction, let us assume that $\overline{T}_0 \ge \frac{11}{6}\bar{T}$. To start with, we first show the following weaker bound
    \begin{equation}
        \label{est:weak}
        E_i'(t) > 0,\quad t \in [0,\frac{11}{6}\bar{T}),\; i = 1,2,
    \end{equation}
    given the scale $R$ chosen suitably small. To prove \eqref{est:weak}, we let $\tau \ge 0$ be the first time instance such that either $E_1' = 0$ or $E_2' = 0$. We may further assume that $\tau \le \frac{11}{6}\bar T$, as otherwise we are done. For all $t \in [0,\tau)$, we have the moment bounds:
    \begin{equation}\label{apriorimoment}
    |E_j(t)| = -E_j(t) \le -E_j(0) = |E_j(0)| = |\bar E_j(0)|.
    \end{equation}
    With the above estimate and following a similar argument to that in Theorem \ref{thm:overlap}, one obtains:
    \begin{equation}\label{est:I2coer}
    I_2[\tsig_+] \ge c_2,\quad t \in [0,\tau).
    \end{equation}
    By a very similar argument to that deriving \eqref{est:E2coercive}, one also obtains the following coercive bound for $E_1'$:
    \begin{equation}\label{est:I1coer}
    I_1[\tsig_+] \ge c_1, \quad t \in [0,\tau),
    \end{equation}
    where $c_1 > 0$ is a constant only depending on $\bar E_1(0)$ and $\bar E_2(0)$. To estimate $J_i[\tsig]$, we use the following bound for the kernel $\calK$: there exists a universal constant $R_0 > 0$ such that
    \begin{equation}\label{est:calK}
    |\calK(r)| \le 1,\quad r \in [0,R_0].
    \end{equation}
    In fact, $\calK$ has the asymptotics $|\calK(r)| \sim r |\log r| + O(r)$ near $r = 0$ (see e.g. \cite{bessel1, bessel}), but the crude bound \eqref{est:calK} suffices. For $t \in [0,\tau) \subset [0, 2\bar T]$, $z, w \in \supp(\tsig_\pm(\cdot,t))$, from \eqref{est:suppbd} we have $|z - w| \le 2d_{\bar T}$. Choosing $R > 0$ sufficiently small that $R|z-w| \le 2d_{\bar T} R < R_0$, we have $|\calK(R(z-w))| \le 1$. Hence for $i = 1,2$, $t \in [0,\tau)$:
    \begin{align*}
        |J_i[\tsig]| &\le \frac12\int_{\supp(\tsig_\pm) \times \supp(\tsig_\pm)} |\tsig_+(w) + \tsig_-(w)| |\tsig_+(z)| dwdz\\
        &\le \|\tsig_\pm\|_{L^1} \|\tsig_\pm\|_{L^1}\le \|\tilde \tau\|_{L^1}^2.
    \end{align*}
    The estimate above together with \eqref{est:I1coer} and \eqref{est:I2coer} yields that
    $$
    E_j'(t) \ge c_j - R\|\tilde\tau\|_{L^1}^2 > \frac{3}{4} c_j,\quad t \in [0,\tau)
    $$
    once we choose $R$ sufficiently small. Thus, we must have $\tau > \frac{11}{6}\bar T$ and \eqref{est:weak} is proved.

    Now with the weaker bound \eqref{est:weak} established, the bound \eqref{apriorimoment} actually holds for all $t \in [0,\frac{11}{6}\bar T)$. Then the same argument to that proving \eqref{est:weak} actually yields the following stronger bound:
    \begin{equation}
        \label{est:strong}
        E_j'(t) > \frac34 c_j, \quad t \in[0,\frac{11}{6}\bar T),\; j = 1,2.
    \end{equation}
    \noindent\textbf{Step 3. Proof of merger.}
    Since $\frac{4}{3} \bar T < \frac{11}{6} \bar T$, we recall $\bar T = \frac{|\bar E_2(0)|}{c_2}$ and use \eqref{est:strong} to deduce that
    \begin{align*}
        E_2(\frac{4}{3} \bar T) &= \bar E_2(0) + \int_0^{\frac{4}{3} \bar T} E'_2(s) ds > \bar E_2(0) +  c_2 \bar T\\
        & = \bar E_2(0) + |\bar E_2(0)| \ge 0,
    \end{align*}
    where we used $\bar E_2(0) < 0$ in the final inequality. However, this is absurd since $E_2(t) \le 0$ for all $t \ge 0$. Hence, we must have $\bar T_0 < \frac{11}{6} \bar T < \infty$. Now, an identical topological argument to that described in the proof of Theorem \ref{thm:overlap} yields the desired result.
\end{proof}

We remark that the proof of Corollary \ref{cor:topo} only utilizes Theorem \ref{thm:overlap}, the symmetry property \eqref{oddodd}, and that the topological properties of the level sets of $\sigma_\pm$ are preserved along the evolution. In the case of the full system \ref{eq:fullsys}, this system also preserves the symmetry \eqref{oddodd} as well as the topological properties of level sets, the latter of which holds because the additional term is a transport term driven by a smooth vector field. Hence, an identical argument to that in Corollary \ref{cor:topo} yields the following statement:
\begin{cor}\label{cor:topofull}
    There exists smooth, compactly supported solutions $\sigma_\pm$ to \eqref{eq:fullsys} such that $F = \frac{\sigma_+ + \sigma_-}{2}$ has the following properties: there exists $T > 0$ such that
    \begin{enumerate}
        \item $\Int(\supp(F(\cdot,0)))$ has 4 connected components, but $\Int(\supp(F(\cdot,T)))$ has 2 connected components;
        \item the following trichotomy holds:
        \begin{enumerate}
            \item $\{|F(x,0)| > \frac12\} = \emptyset$, but $\{|F(x,T)| > \frac12\} \neq \emptyset$, or;
            \item $Card(\{|F(x,0)| = \frac12\}) = 4$, but $Card(\{|F(x,T)| = \frac12\}) > 4$, or;
            \item There exists $|\alpha| \in (0,\frac12)$ such that $\alpha$ is a regular value of $F(x,0)$, but a critical value of $F(x,T)$.
        \end{enumerate}
    \end{enumerate}
\end{cor}
\section{Merger stability}\la{mergstabi}
In this section we compare the dynamics of \eqref{sigmaeq} to the dynamics when $\mathbb B$ is replaced by
$-\mathbb U$. That is, we compare the original screened right-handed model to the unscreened one. In this latter case, in view of \eqref{Vpm},  we have the system
\be
\pa_t \sigma_{+} - \mathbb U(\sigma_{-})\cdot\na \sigma_{+} = 0
\la{sigpleq}
\ee
and
\be
\pa_t \sigma_{-} + \mathbb U(\sigma_{+})\cdot\na\sigma_{-} = 0.
\la{sigmineq}
\ee
We consider the same fixed smooth and compactly supported initial data $\sigma_{\pm}(0)= \tau_{\pm} $ and we consider a situation in which at time zero the initial data have disjoint supports, and in particular,
\be
\tau_+(x)\tau_-(x)= 0,
\la{timezero}
\ee
and at time $T$ we have
\be
\int_{\Rr^2}\sigma_+(x,T)\sigma_-(x,T)dx \ge \gamma>0.
\la{timeone}
\ee
We take now solutions $\widetilde\sigma_{\pm}$ of the screened system \eqref{sigmaeq} with the same initial data and note that
\be
\ba
\left |\int_{\Rr^2} (\widetilde\sigma_+(x,t)\widetilde\sigma_-(x,t) -\sigma_+(x,t)\sigma_{-}(x,t))dx\right | \\
=\left |\int_{\Rr^2}(\delta_+(x,t)\widetilde\sigma_-(x,t) +\delta_-(x,t)\sigma_+(x,t))dx\right|\\
\le \|\delta_+(t)\|_{L^p}\|\tau_-\|_{L^{p'}} + \|\delta_-(t)\|_{L^p}\|\tau_+\|_{L^{p'}}
\ea
\la{prodb}
\ee
where
\be
\delta_{\pm} = \widetilde\sigma_{\pm} -\sigma_{\pm}
\la{deltas}
\ee
and where $p'$ is the conjugate exponent of $p$. We used the fact that both the screened and the unscreened systems conserve $L^{p'}$ norms. We take $p<2$ and we estimate the norms  $\|\delta_{\pm}(t)\|_{L^p}$.

Let us denote 
\be
\mathbb S = \mathbb B + \mathbb U = \na^{\perp} (\mathbb I -\Delta)^{-1}\Delta^{-1},
\la{mathbbS}
\ee
and observe that the equations \eqref{Vpm}
can be written as
\be
\left \{
\ba
\mathbb V_{+} (\sigma) = - \mathbb U(\sigma_-) + \mathbb S\left(\fr{\sigma_+ + \sigma_-}{2}\right)\\
\mathbb V_-(\sigma) =   \mathbb U(\sigma_+) - \mathbb S\left(\fr{\sigma_+ + \sigma_-}{2}\right),\\
\ea
\right .
\la{VpmS}
\ee
The equations obeyed by $\widetilde\sigma_{\pm}$ are thus
\be
\pa_t\widetilde\sigma_+ + \left[- \mathbb U(\widetilde\sigma_-) + \mathbb S(\ov\sigma)\right ]\cdot\na\widetilde\sigma_+ =0,
\la{widesigmapluseq}
\ee
and
\be
\pa_t\widetilde\sigma_- + \left[ \mathbb U(\widetilde\sigma_+) - \mathbb S(\ov\sigma)\right ]\cdot\na\widetilde\sigma_- =0,
\la{widesigmaminuseq}
\ee
where 
\be
\ov\sigma = \fr{1}{2}(\widetilde\sigma_+ + \widetilde\sigma_-)
\la{ovsigma}
\ee
The equations obeyed by $\delta_{\pm}$, in view of \eqref{sigpleq}, \eqref{sigmineq}, 
\eqref{widesigmapluseq} and \eqref{widesigmaminuseq},
\be
\pa_t \delta_+  + \left[- \mathbb U(\widetilde\sigma_-) + \mathbb S(\ov\sigma)\right ]\cdot\na \delta_+
   -\mathbb U(\delta_-)\cdot\na\sigma_+ + \mathbb S(\ov\sigma)\cdot\na\sigma_+ = 0,
\la{deltapluseq}
\ee
and
\be
\pa_t \delta_-  + \left[ \mathbb U(\widetilde\sigma_+) - \mathbb S(\ov\sigma)\right ]\cdot\na\delta_- + \mathbb U(\delta_+)\cdot\na\sigma_- - \mathbb S(\ov\sigma)\cdot\na\sigma_- = 0,
\la{deltaminuseq}
\ee
Because the second terms in both equations are transport terms with incompressible velocities, we obtain
\be
\fr{d}{dt}\|\delta_+\|_{L^p} \le \|U(\delta_-)\cdot\na\sigma_+\|_{L^p} + \|\mathbb S(\ov\sigma)\cdot\na\sigma_+\|_{L^p}
\la{deltpluslp}
\ee
and 
\be
\fr{d}{dt}\|\delta_-\|_{L^p} \le \|U(\delta_+)\cdot\na\sigma_-\|_{L^p} + \|\mathbb S(\ov\sigma)\cdot\na\sigma_-\|_{L^p}
\la{deltminlp}
\ee

We estimate the terms $\mathbb U(\delta)\na\sigma$ using the Hardy-Littlewood-Sobolev inequality
\be
\|U(\delta)\|_{L^r} \le C\|\delta\|_{L^p}
\la{hls}
\ee
where $\fr{1}{r} = \fr{1}{p}- \fr{1}{2}$ (that is why we took $p<2$). Then, using duality, we have
that $U(\delta)\na\sigma$ is estimated in $L^p$ by 
\be
\|U(\delta)\na\sigma\|_{L^p} \le C\|\delta\|_{L^p}\|\na\sigma\|_{L^2}.
\la{goodb}
\ee
We also estimate the terms $\|\mathbb S(\ov\sigma)\na\sigma\|_{L^p}$
\be
\|\mathbb S(\ov\sigma)\na\sigma\|_{L^p}\le \|\na\sigma\|_{L^2} \|\chi\mathbb S(\ov\sigma)\|_{L^r}
\la{forcetermlp}
\ee
where $\chi$ is the characteristic (indicator function) of the support of $\sigma$.
We estimate the kernel of $\mathbb S$. Because its Fourier transform is integrable, it follows that the kernel of $\mathbb S$ is bounded.  
Using the fact that the kernel is bounded, we have
\be
\|\mathbb S(\ov\sigma)\|_{L^{\infty}}\le C\|\ov\sigma\|_{L^1} = C\|\tau\|_{L^1}
\la{sovsiglifty}
\ee
and thus
\be
\|\mathbb S(\ov\sigma)\na\sigma\|_{L^p} \le C\|\na\sigma\|_{L^2}\|\ov\sigma\|_{L^1}\|\chi\|_{L^r}
\la{boundf}
\ee
Now  we have that 
\be
\|\chi(t)\|_{L^r} \le  A^{\fr{1}{r}}
\la{chitr}
\ee
where $A$ is the area of the support of $\tau$.
We have thus
\be
\fr{d}{dt}\|\delta(t)\|_{L^p} \le \|\na \sigma(t)\|_{L^2}\|\delta(t)\|_{L^p} + C\|\na\sigma(t)\|_{L^2}A^{\fr{1}{p}-\fr{1}{2}}\|\tau\|_{L^1}
\la{deltalpev}
\ee
where we use  $\|\na\sigma\|_{L^2} = \max \{\|\na\sigma_{+}\|_{L^2};  \|\na\sigma_{-}\|_{L^2}\}$ and
$\|\delta\|_{L^p} = \|\delta_+\|_{L^p} + \|\delta_-\|_{L^p}$.
We have proved thus
\beg{prop}\la{stab}
Let $\sigma_\pm(t)$ be solutions of the unscreened right-handed system equations \eqref{sigpleq}, \eqref{sigmineq} with smooth compactly supported initial data $\tau_{\pm}$. Let $\widetilde\sigma_{\pm}(t)$
be the solutions of the screened right-handed system \eqref{widesigmapluseq}, \eqref{widesigmaminuseq} with the same initial data. Let $1<p<2$ and let $A$ denote the maximum of the areas of the supports of
$\tau_{\pm}$. 
Then 
\be
\|\widetilde\sigma_{\pm}(t)-\sigma_{\pm}(t)\|_{L^p} \le CA^{\fr{1}{p}-\fr{1}{2}}\|\tau\|_{L^1}
\int_0^t\|\na \sigma (s)\|_{L^2}e^{\int_s^t\|\na\sigma(z)\|_{L^2}dz}ds
\la{lpboundiff}
\ee
\end{prop}
Now we take advantage of the fact that unscreened equations have scaling invariance. Thus, if $\tau_{\pm}$  are functions with length scale $\epsilon$, meaning that they are rescaling $\tau_{\pm}(x) = f_{\pm}\left(\fr{x}{\epsilon}\right)$, with $f$ smooth bounded functions with bounded derivatives and compactly supported in a domain unit size, then 
\be
A^{\fr{1}{p}-\fr{1}{2}}\|\tau\|_{L^1} \sim \epsilon^{\fr{2}{p} +1}
\la{forcesize}
\ee
and, in view of \eqref{nasigmabound} and the volume conservation, we have
\be
\|\na\sigma(t)\|_{L^2} \le \epsilon^{-1+1}e^{Ce^{Ct}} = e^{Ce^{Ct}}
\la{nasigmaunif}
\ee
 is bounded independently of epsilon. From \eqref{deltalpev}, in view of the fact that the initial data for $\delta $ vanish, and the uniform in $\epsilon$ bound \eqref{nasigmaunif}, we have that, for any $T$, there exists
 a constant $C(T)$ independent of $\epsilon$, such that,
\be
\|\widetilde\sigma_{\pm}(t) - \sigma_{\pm}(t)\|_{L^p} \le \epsilon^{\fr{2}{p} +1}C(T)
\la{stabest}
\ee
holds for all $0\le t\le T$.  In particular,  because the lower bound $\gamma$ of \eqref{timeone} is of the order $\epsilon^2$, $\gamma\ge c\epsilon^2$, then we see from \eqref{prodb} and \eqref{stabest} that
\be
\int_{\Rr^2}\widetilde\sigma_+(x,T)\widetilde\sigma_-(x,T)dx \ge c\epsilon^2 -
C(T)\epsilon^{\fr{2}{p'} + \fr{2}{p} +1} = c\epsilon^2- C(T)\epsilon^3
\la{voila}
\ee
In view of the fact that we are free to choose $\epsilon$ and observe the unscreened merger at time $T$ we see that the screened evolution must also lead to merger.

\beg{thm}\la{stabmerg} Assume that there exist smooth compactly supported initial data $\tau_{\pm}$ of the unscreened right-handed system \eqref{sigpleq}, \eqref{sigmineq} that have disjoint supports, and such that there exists $T>0$  such that \eqref{timeone} holds. Then there exist smooth compactly supported initial data $\widetilde\tau_{\pm}$ of the original screened right-handed system \eqref{widesigmapluseq}, \eqref{widesigmaminuseq} whose supports are disjoint and whose solutions $\widetilde\sigma_{\pm}(t)$ obey  at time $T$
\be
\int_{\Rr^2}\widetilde\sigma_+(x,T)\widetilde\sigma_-(x,T)dx >0.
\la{mergwide}
\ee
\end{thm}
\beg{proof} We take $\widetilde\tau_{\pm} = \tau_{\pm}\left (\fr{x}{\epsilon}\right)$ and choose $\epsilon$ small enough.
\end{proof}

\beg{rem}\la{likewise} The exact same result holds for the left-handed equations, with the same proof.
\end{rem}

We compare also the Lagrangian trajectories $X_{\pm}$ to the Lagrangian trajectories $Y_{\pm}$ of the system
\eqref{sigpleq}-\eqref{sigmineq}.  We note that 
\be
\sigma_+ = \tau_{+}\circ Y_+^{-1}
\la{sigpYmin}
\ee
and
\be
\sigma_- = \tau_-\circ Y_{-}^{-1}
\la{sigminYpl}
\ee
where
\be
\pa_t Y_{+} = -\mathbb U\left (\tau_-\circ Y_-^{-1}\right)\circ Y_{+}
\la{Ypleq}
\ee
and
\be
\pa_t Y_{-} = \mathbb U\left (\tau_+\circ Y_+^{-1}\right)\circ Y_{-}
\la{Ymineq}
\ee

Now \eqref{Xeq} can be written as
\be
\pa_t X_+ = -\mathbb U\left (\tau_-\circ X_-^{-1}\right)\circ X_{+}  + \fr{1}{2}\mathbb S\left(\tau_+\circ X_{+}^{-1} + \tau_-\circ X_-^{-1}\right)\circ X_{+}
\la{XplSeq}
\ee
and
\be
\pa_t X_- = \mathbb U\left (\tau_+\circ X_+^{-1}\right)\circ X_{-}  - \fr{1}{2}\mathbb S\left(\tau_+\circ X_{+}^{-1} + \tau_-\circ X_-^{-1}\right)\circ X_{-}
\la{XminSeq}
\ee
The functions $\tau_{\pm}$ are smooth, time independent and compactly supported. We note that
 $X= (X_+, X_-)$ and $Y = (Y_+, Y_-)$  starting from the identity maps, are defined and smooth for all time. 
 We aim to give bounds in  situations in which  and $X_\pm$ remain close respectively to $Y_\pm$ although $Y_+$ and $Y_-$ are not close. We consider the differences
\be
Z_{\pm}(a,t) = X_{\pm}(a,t)- Y_{\pm}(a,t)
\la{Zpm}
\ee
and their evolution
\be
\pa_t Z_{+} = \mathbb U\left (\tau_-\circ Y_-^{-1}\right)\circ Y_{+} -\mathbb U\left (\tau_-\circ X_-^{-1}\right)\circ X_{+} + \fr{1}{2}\mathbb S\left(\tau_+\circ X_{+}^{-1} + \tau_-\circ X_-^{-1}\right)\circ X_{+}
\la{Zpleq}
\ee
and
\be
\pa_t Z_{-} =  \mathbb U\left (\tau_+\circ X_+^{-1}\right)\circ X_{-} -\mathbb U\left (\tau_+\circ Y_+^{-1}\right)\circ Y_{-} - \fr{1}{2}\mathbb S\left(\tau_+\circ X_{+}^{-1} + \tau_-\circ X_-^{-1}\right)\circ X_{-}
\la{Zmineq}
\ee
We define for $s\in [0,1]$
\be
X_{\pm}^s(a,t) = s X_{\pm}(a,t) + (1- s) Y_{\pm}(a,t).
\la{Xmpms}
\ee
We have
\be 
Z_{\pm}(a,t) = \fr{d}{ds} X^s_{\pm}(a,t)
\la{Zds}
\ee
and 
\be
\pa_t Z_{+} = -\int_0^1 \fr{d}{ds}\left(\mathbb U\left (\tau_-\circ (X^s_-)^{-1}\right)\circ X_+^s\right) ds 
+ \fr{1}{2}\mathbb S\left(\tau_+\circ X_{+}^{-1} + \tau_-\circ X_-^{-1}\right)\circ X_{+},
\la{Zplds}
\ee
\be
\pa_tZ_- = \int_0^1\fr{d}{ds}\left(\mathbb U\left (\tau_+\circ (X_+^s)^{-1}\right)\circ X_{-}^s\right)ds 
 - \fr{1}{2}\mathbb S\left(\tau_+\circ X_{+}^{-1} + \tau_-\circ X_-^{-1}\right)\circ X_{-}
\la{Zminds}
\ee
We consider the terms with $\mathbb S$ as forcing terms and show that the $\mathbb U$ terms lead to a controlled Lipschitz evolution. Let us denote
\be
\eta_{\pm} = Z_{\pm}\circ (X^s_{\pm})^{-1}
\ee
We have
\be
-\fr{d}{ds}\left( \mathbb U\left (\tau_-\circ (X^s_-)^{-1}\right)\circ X_+^s\right)  = 
\mathbb U(\eta_-\cdot\nabla(\tau_-\circ (X^s_-)^{-1})\circ X_+^s  - (\na\mathbb U(\tau_-\circ (X^s_-)^{-1})\circ X_+^s)Z_+
\la{dsUplus}
\ee
and 
\be
\fr{d}{ds} \left(\mathbb U\left (\tau_+\circ (X_+^s)^{-1}\right)\circ X_{-}^s\right) = (\na\mathbb U(\tau_+\circ (X^s_+)^{-1})\circ X_-^s)Z_- -\mathbb U(\eta_+\cdot\nabla(\tau_+\circ (X^s_+)^{-1})\circ X_-^s.
\la{dsUminus}
\ee
Composing from the right in \eqref{dsUplus} with $(X_+^s)^{-1}$ and in \eqref{dsUminus} with $(X_{-}^s)^{-1}$ we have
\be
 -\left [\fr{d}{ds}\left( \mathbb U\left (\tau_-\circ (X^s_-)^{-1}\right)\circ X_+^s\right) \right] \circ (X^s_+)^{-1} =
 \mathbb U(\eta_-\cdot\nabla \sigma_-^s) - \eta_+\cdot\na \mathbb U(\sigma_-^s)
 \la{dsUminuseu}
 \ee
 and
 \be
 \left [\fr{d}{ds} \left(\mathbb U\left (\tau_+\circ (X_+^s)^{-1}\right)\circ X_{-}^s\right)\right]\circ (X_-^s)^{-1}
= \eta_-\cdot\na \mathbb U(\sigma_+^s) - \mathbb U(\eta_+\cdot \na\sigma_+^s)
\la{dsUpluseu}
\ee
 where 
 \be
 \sigma_{\pm}^s = \tau_{\pm}\circ(X^s_{\pm})^{-1}.
 \la{sigmaspm}
 \ee
 Note the fact that if $\eta_-$ and $\eta_+$ would be the same, then we would obtain a commutator structure.
 This is related to the important commutator structure in single fluid equations \cite{hydro}, when there is no distinction between $Y_+$ and $Y_-$. That commutator structure is the reason why for single fluid equations the Lagrangian path evolution is Lipschitz in path space without loss of derivatives. Here however we do loose derivatives because $Y_+$ and $Y_-$ are not the same.

 We consider an arbitrary time interval $[0,T]$ and take full advantage of the fact that we have uniform a priori information on $\sigma_{\pm}$, and on both $Y_{\pm}$ and $X_{\pm}$ and their inverses. In particular,  we have the following a priori information. From \eqref{naVlifty} and with $A(t)$ defined in \eqref{At},
 \be
 A(t) = \exp\left[ C\left( e^{Ct\|\tau\|_{L^{\infty}}} -1\right )\log\left (2 + \fr{\|\na\tau\|_{L^{\infty}}\|\tau\|_{L^1}^{\fr{1}{2}}}{\|\tau\|_{L^{\infty}}^{\fr{3}{2}}}\right)\right]
\la{Atagain}
 \ee
 we have
 \be
 \|\na \mathbb U(\sigma_{\pm}^s(t))\|_{L^{\infty}}\le C\|\tau\|_{L^{\infty}}A(t)
 \la{naUsigmab}
 \ee
 and  with \eqref{nasigmalpbA}
 \be
 \|\na\sigma_{\pm}^s(t)\|_{L^p} \le A(t) \|\na\tau_{\pm}\|_{L^p} 
 \la{nasigmalpagain}
 \ee
 In order to track the $L^{\infty}$ norm of $Z$ we estimate separately using the kernels of $\mathbb U$, and $\mathbb S$. 
We have directly form \eqref{naUsigmab}
\be
\|\eta_+\cdot\na \mathbb U(\sigma_-^s)\|_{L^{\infty}}  \le C\|\eta_+\|_{L^{\infty}}\|\tau\|_{L^{\infty}}A(t)
\la{gradUeta+}
\ee
and
\be
 \|\eta_-\cdot\na \mathbb U(\sigma_+^s)\|_{L^{\infty}}\le  C\|\eta_-\|_{L^{\infty}}\|\tau\|_{L^{\infty}}A(t)
\la{gradUeta-}
\ee
The terms $\mathbb U(\eta_{\pm}\na \sigma_{\pm}^s)$ are estimated using an upper bound on the kernel of $\mathbb U$. We use a useful extrapolation inequality for the Hardy-Littlewood-Sobolev end point estimate, 
\be
\left | \int_{\R^2}\fr{g(y)}{|x-y|}dy\right | \le C\|g\|_{L^2}\left [1 + \log\left(\fr{\|g\|_{L^1}\|g\|_{L^{\infty}}}{\|g\|_{L^{2}}^2}\right)\right].
\la{extrahls}
\ee
The proof of this estimate follows directly by splitting the integral in regions $|x-y|\le l$, $l\le |x-y|\le L$ and
$|x-y|\ge L$ with $l = \fr{\|g\|_{L^2}}{\|g\|_{L^{\infty}}}$ and $L = \fr{\|g\|_{L^1}}{\|g\|_{L^2}}$ and noticing that always $l\le L$. We apply this with $g= |\na\sigma_{\pm}^s(t)|$ and obtain
\be
\|\mathbb U(\eta_{\pm}\na \sigma_{\pm}^s(t)\|_{L^{\infty}} \le CA(t)\|\eta_{\pm}(t)\|_{L^{\infty}}\|\na\tau\|_{L^2}
\left [1 + \log\left(\fr{\|\na\tau\|_{L^1}\|\na\tau\|_{L^{\infty}}}{\|\na\tau\|_{L^{2}}^2}\right)\right].
\la{Uetagradb}
\ee
We did this in order to obtain a scale invariant bound. This means that the bound depends on the initial norms with a loss of derivative, but the factors are scale invariant, and uniform in time on $[0,T]$. 
Fixing $T$ for simplicity of exposition and putting
\be
\Gamma(T) = CA(T)\|\tau\|_{L^{\infty}}\left[1 + \fr{\|\na\tau\|_{L^2}}{\|\tau\|_{L^{\infty}}} \log\left(\fr{\|\na\tau\|_{L^1}\|\na\tau\|_{L^{\infty}}}{\|\na\tau\|_{L^{2}}^2}\right)\right],
\la{GammaT}
\ee
we have 
\be
\|Z(t)\|_{L^{\infty}} \le \Gamma(T)\int_0^t\|Z(\nu)\|_{L^{\infty}}d\nu + \int_0^1ds\int_0^t\|\mathbb S(\sigma_{\pm}^s(\nu))\|_{L^{\infty}}d\nu
\la{beforegr}
\ee
Now the kernel of $\mathbb S$ is bounded and continuous. Using it we have that
\be
\|\mathbb S(\sigma_+^s + \sigma_-^s)\|_{L^\infty} \le C\|\tau\|_{L^1}.
\la{Sftybound}
\ee
This is not scale invariant, because of the screening. Returning to \eqref{beforegr} we proved
\beg{prop}\la{linftymerger}
Let $\tau_{\pm}$ be smooth and compactly supported initial data. Let $T>0$ be a fixed time. There exists a scale invariant constant $\Gamma(T)$ depending on norms $\|\tau\|_{W^{1,p}}$, $1\le p\le\infty $,  such that difference of Lagrangian paths obeys 
\be
\sup \|X_{\pm}(t)-Y_{\pm}(t)\|_{L^{\infty}}\le C\|\tau\|_{L^1} e^{T\Gamma(T)}
\la{supZ}
\ee
for $0\le t\le T$.
\end{prop}
The main point here is that if we consider a family  $\tau_{\pm} = f_{\pm}(\fr{x}{\epsilon})$ like before, then 
$\Gamma(T)$ is {\em{independent  of $\epsilon$}} while $\|\tau\|_{L^1}$ scales like $\epsilon^2$.
 \be
 \|Z(t)\|_{L^{\infty}} \le Ce^{T\Gamma(T)}\epsilon^2.
 \la{Zftyeps}
 \ee
 This bound implies that the distance between Lagrangian trajectories is significantly less than $\epsilon$. 
The unscreened merger occurs at scale $\epsilon$ at times of order one.

 \section{The limit of zero resistivity}\la{invlim}
 
 We investigate here resistive regularization and the zero resistivity limit in a subcritical regime. 
  The viscous version to the system \eqref{sigmaeq}--\eqref{Vpm} is written as follows:
 \begin{equation}
     \label{sigmaeqvis}
     \pa_t \sigma_{\pm} + \mathbb V_{\pm}(\sigma)\cdot\na \sigma_{\pm} - \nu_\pm \Delta \sigma_\pm=0, \quad \nu_\pm \ge 0,
 \end{equation}
 where we recall the relation
 $$
 \mathbb V_{\pm}(\sigma) = \fr{1}{2}\left (\mathbb U (\sigma_{+}-\sigma_{-}) \pm \mathbb B(\sigma_+ +\sigma_-)\right ).
 $$
The goal of this section is to establish an inviscid limit result concerning sufficiently regular solutions of \eqref{sigmaeqvis} in Sobolev spaces. As a preliminary, we first establish the following global well-posedness result for \eqref{sigmaeqvis} in subcritical Sobolev spaces:
\begin{thm}
    Let $s > 1$ and $\tau_\pm \in (H^s \cap L^1)(\R^2)$. For any $\nu_\pm \ge 0$, there exists a unique solution $\sigma_\pm \in C([0,\infty); (H^s \cap L^1)(\R^2))$ to \eqref{sigmaeqvis} with initial data $\sigma_\pm(x,0) = \tau_\pm(x)$ such that for any $t \ge 0$,
    \begin{equation}
        \label{est:visenergy}
        \|\sigma(t)\|_{H^s} \le \|\tau\|_{H^s} e^{Ce^{Ct}},
    \end{equation}
    where $C > 0$ is a constant that is independent of $\nu_\pm$.
\end{thm}
\begin{proof}
    We provide suitable a priori estimates. Concerning a priori $L^1$ bounds, we have $\|\sigma(t)\|_{L^1} \le \|\tau\|_{L^1}$: when $\nu_\pm > 0$, since $\V_\pm$ is divergence-free, the claimed bound follows from the standard fact that the solution map to the advection-diffusion equation with divergence-free velocity is a $L^1$ contraction. When $\nu_\pm = 0$, we actually have conservation of $\|\sigma\|_{L^1}$ due to transport by a volume-preserving diffeomorphism.
    
    Let us consider $s > 1$ to be an integer for simplicity. Also, we denote $\na^i$ to be a generic $i$--th order spatial derivative. Fix $0 \le i \le s$. Differentiating \eqref{sigmaeqvis} by $\na^i$, testing the differentiated equation by $\nabla^i \sigma_\pm$, and using incompressibility, we have
    \begin{align*}
        \frac12 \frac{d}{dt}\|\nabla^i \sigma_\pm\|_{L^2}^2 + \nu_\pm \|\nabla^i \nabla\sigma_\pm\|_{L^2}^2 = -\int_{\R^2} [\na^i, \mathbb V(\sigma)\cdot \nabla] \sigma_\pm \nabla^i \sigma_\pm dx.
    \end{align*}
    Familiar calculus and commutator estimates yield:
    \begin{align*}
        \frac12 \frac{d}{dt}\|\nabla^i \sigma_\pm\|_{L^2}^2 + \nu_\pm \|\nabla^i \nabla\sigma_\pm\|_{L^2}^2 \le C(\|\nabla \mathbb V(\sigma)\|_{L^\infty} + \|\sigma\|_{L^\infty})\|\nabla^i \sigma_\pm\|_{L^2}^2.
    \end{align*}
    Summing over $i= 0,\hdots, s$ and an application of Gr\"onwall inequality yield that, for any $\nu_\pm \ge 0$, $t \ge 0$:
    \begin{equation}
        \label{energy1}
        \|\sigma_\pm\|_{H^s}^2 \le \|\tau_\pm\|_{H^s}^2 \exp\left(C\int_0^t \|\nabla \mathbb V(\sigma)(s)\|_{L^\infty} + \|\sigma(s)\|_{L^\infty} ds\right).
    \end{equation}
    Note that by Sobolev embedding, there exists $\alpha(s) > 0$ such that $\sigma_\pm \in (C^\alpha \cap L^1)(\R^2)$. Then an application of the inequality \eqref{extra} and conservation of $\|\sigma_\pm\|_{L^\infty}$ to \eqref{energy1} immediately yields the desired estimate \eqref{est:visenergy}.
\end{proof}

\begin{thm}
    Let $s > 2$ and $\tau_\pm \in H^s(\R^2)$ be compactly supported. Let $\bar\sigma_\pm \in C([0,\infty);H^s \cap L^1)$ ($\sigma_\pm \in C([0,\infty);H^s \cap L^1)$) be the global solution to the inviscid problem \eqref{sigmaeq} (viscous problem \eqref{sigmaeqvis}) with initial datum $\tau_\pm$. Then the following statements hold: for any $T \ge 0$, $t \in [0,T]$, $\nu_\pm < 1$, we have
    \begin{enumerate}
        \item $\|\sigma_\bullet - \bar\sigma_\bullet\|_{H^{s-1}\cap L^1} \le C(T)((\nu_+ + \nu_-) t)^\frac12 $, $\bullet = +,-$;
        \item For any $s' \in (s-1,s)$, $\|\sigma_\bullet - \bar\sigma_\bullet\|_{H^{s'}\cap L^1} \le C(T)((\nu_+ + \nu_-) t)^\frac{s-s'}{2} $, $\bullet = +,-$.
    \end{enumerate}
\end{thm}
\begin{proof}
    Let $T \ge 0$ be fixed. First we observe that, by \eqref{est:visenergy}, there exists $C(T) > 0$ such that
    \begin{equation}
        \label{est:uniform}
        \|\sigma\|_{H^s} + \|\bar\sigma\|_{H^s} \le C(T).
    \end{equation}
    Denote $d_\pm = \sigma_\pm - \bar\sigma_\pm$. Then the difference $d_\pm$ satisfies the following equation:
    \begin{equation}
        \label{eq:dpm}
        \pt d_\pm + \mathbb V_{\pm}(\sigma)\cdot\na d_{\pm} + \mathbb V_{\pm}(d)\cdot\na \bar\sigma_{\pm} - \nu_\pm \Delta \sigma_\pm = 0.
    \end{equation}
    equipped with initial datum $d_\pm(x,0) = 0$.
    
    For simplicity, let us consider $s > 2$ be an integer. For any $0 \le i \le s-1$, we differentiate \eqref{eq:dpm} by $\na^i$ and test by $\na^i d_\pm$. It holds that
    \begin{align*}
        \frac12 \frac{d}{dt} \|\na^i d_\pm\|_{L^2}^2 + \nu_\pm \|\na^i\na d_\pm\|_{L^2}^2 &= -\int_{\R^2} [\na^i, \mathbb V_\pm(\sigma)\cdot \nabla] d_\pm \nabla^i d_\pm dx\\
        &\quad -\int_{\R^2} \na^i(\mathbb{V}_\pm(d) \cdot \na \bar\sigma_\pm) \na^i d_\pm dx\\
        &\quad + \nu_\pm \int_{\R^2} \na(\na^i \bar\sigma_\pm) \cdot \na(\na^i d_{\pm}) dx\\
        &=: I_1 + I_2 + I_3.
    \end{align*}
    Using calculus estimate and Sobolev embedding, we have
    $$
    |I_1| \le C \|\na \mathbb V_\pm(\sigma)\|_{L^\infty} \|\na^id_\pm\|_{L^2}^2 \le C\|\sigma\|_{H^s}\|\na^id_\pm\|_{L^2}^2.
    $$
    To estimate $I_2$, we see that
    \begin{align*}
        |I_2| &\le \int_{\R^2} |\V_\pm(d) \cdot \nabla^{i+1}\bar\sigma_\pm \na^i d_\pm| dx + \sum_{j=1}^i |\na^j \V_\pm (d) \cdot \na^{i-j+1}\bar\sigma_\pm \na^i d_\pm | dx =: I_{21} + I_{22}.
    \end{align*}
    Since $s - 1 > 1$, we note that 
    \begin{equation}\label{Vestvisc}
    \|\V_\pm(d)\|_{L^\infty} \le C\|d\|_{L^1}^{\frac12}\|d\|_{L^\infty}^\frac12 \le C\|d\|_{L^1}^{\frac12}\|d\|_{H^{s-1}}^\frac12.
    \end{equation}
    Then we may bound $I_{21}$ by:
    $$
    I_{21} \le C\|\bar\sigma\|_{H^s}\|d_\pm\|_{L^1}^{\frac12}\|d_\pm\|_{H^{s-1}}^\frac32 \le C\|\bar\sigma\|_{H^s}\left(\|d_\pm\|_{L^1}^2 + \|d_\pm\|_{H^{s-1}}^2\right).
    $$
    $I_{22}$ can be estimated similarly to $I_1$ by using calculus and commutator estimate, Sobolev embedding, and a Calder\'on-Zygmund type estimate:
    \begin{align*}
    I_{22} &\le C\|\na \V_\pm(d)\|_{L^\infty} \|\bar\sigma\|_{H^{s-1}} \|\na^i d_\pm\|_{L^2}\\
    &\le C\|\na \V_\pm(d)\|_{H^{s-1}}\|\bar\sigma\|_{H^{s-1}}\|\na^i d_\pm\|_{L^2}\\
    &\le C\|\bar\sigma\|_{H^{s-1}}\|d_\pm\|_{H^{s-1}}^2.
    \end{align*}
    Finally, to estimate $I_3$, we have
    \begin{align*}
        |I_3| \le \nu_\pm \|\bar\sigma_\pm\|_{H^s} \|\na^{i}\na d_\pm\|_{L^2} \le \frac12 \nu_\pm \|\na^{i}\na d_\pm\|_{L^2}^2 + C\nu_\pm \|\bar\sigma_\pm\|_{H^s}^2,
    \end{align*}
    Combining the above estimates, summing over all $i = 0,\hdots, s-1$, and invoking \eqref{est:uniform}, we obtain the following inequality concerning $\|d_\pm\|_{H^{s-1}}$:
    \begin{equation}\label{est:dHs-1}
    \begin{aligned}
        \frac{d}{dt} \|d_\pm\|_{H^{s-1}}^2 + \nu_\pm \|d_\pm\|_{H^{s}}^2 &\le C\left[\left(\|\sigma\|_{H^s} + \|\bar\sigma\|_{H^s}\right) (\|d_\pm\|_{H^{s-1}}^2 + \|d_\pm\|_{L^1}^2)+ \nu_\pm \|\bar\sigma_\pm\|_{H^s}^2\right]\\
        &\le C(T) (\|d_\pm\|_{H^{s-1}}^2+ \|d_\pm\|_{L^1}^2+ \nu_\pm).
    \end{aligned}
    \end{equation}
    Another important ingredient is an estimate for $\|d_\pm\|_{L^1}$. Rewriting \eqref{eq:dpm} as
    $$
    \pt d_\pm + \V_\pm(\sigma) \cdot \nabla d_\pm - \nu_\pm \Delta d_\pm = - \V_\pm (d) \cdot \na \bar\sigma_\pm - \nu_\pm \Delta \bar\sigma_\pm,
    $$
    we deduce the following estimate:
    \begin{align*}
    \frac{d}{dt}\|d_\pm\|_{L^1} &\le \|\V_\pm (d)\|_{L^\infty}\|\na \bar\sigma_\pm\|_{L^1} + \nu_\pm \|\Delta \bar\sigma_\pm\|_{L^1}\\
    &\le C|\Omega_0|^\frac12\|\na \bar\sigma_\pm\|_{L^2} \|d\|_{L^1}^{\frac12}\|d\|_{H^{s-1}}^\frac12 + \nu_\pm |\Omega_0|^\frac12 \|\Delta\bar\sigma_\pm\|_{L^2}.
    \end{align*}
    Here, $\Omega_0 = \supp(\tau_\pm)$, and we used that the evolution of the inviscid solution $\bar\sigma$ is volume-preserving. A further use of Young's inequality and \eqref{est:uniform} yields:
    \begin{equation}\label{est:dL1}
    \begin{aligned}
        \frac12\frac{d}{dt}\|d_\pm\|_{L^1}^2 &\le C|\Omega_0|^\frac12\|\na \bar\sigma_\pm\|_{L^2} \|d_\pm\|_{L^1}^{\frac32}\|d\|_{H^{s-1}}^\frac12 + \nu_\pm |\Omega_0|^\frac12 \|\Delta\bar\sigma_\pm\|_{L^2}\|d_\pm\|_{L^1}\\
        &\le C|\Omega_0|^\frac12 (\|\bar\sigma_\pm\|_{H^s}+\nu_\pm) \|d_\pm\|_{L^1}^2 + C|\Omega_0|^\frac12 \|\bar\sigma_\pm\|_{H^s} \|d_\pm\|_{H^{s-1}}^2 + \frac12 \nu_\pm |\Omega_0|^\frac12 \|\bar\sigma_\pm\|_{H^s}^2\\
        &\le C(T)|\Omega_0|^\frac12(\|d_\pm\|_{H^{s-1}}^2 + \|d_\pm\|_{L^1}^2 + \nu_\pm).
    \end{aligned}
    \end{equation}
    Here, we also used $s > 2$. Adding \eqref{est:dHs-1} and \eqref{est:dL1}, using $d_\pm (x,0) = 0$, and then Gr\"onwall inequality, we have
    \begin{equation}\label{est:vanishings-1}
    \|d_\pm\|_{H^{s-1}}^2 + \|d_\pm\|_{L^1}^2 \le C(T) ((\nu_+ + \nu_-) t),\quad t \in [0,T],
    \end{equation}
    which finishes the first statement of the Theorem. The second statement follows from \eqref{est:vanishings-1}, the following interpolation inequality:
    $$
    \|d_\pm\|_{H^{s'}} \le \|d_\pm\|_{H^{s-1}}^{s-s'} \|d_\pm\|_{H^{s}}^{s'-s+1},\quad s' \in (s-1,s),
    $$
    and the uniform bound $\|d_\pm\|_{H^s} \le C(T)$ in view of \eqref{est:uniform}.
\end{proof}

 \section{A formal derivation}\la{deriv}
 We offer here a unified derivation of both the right-handed and the left-handed equations. They are derived from a two-fluids action
 \be
 \mathcal A = \int_{t_1}^{t_2} \int_{\R^d}\left [\fr{1}{2}(A u_1(x,t))\cdot  u_1(x,t)  + \fr{1}{2}(Bu_2(x,t))\cdot u_2(x,t) + (Cu_1(x,t)) \cdot u_2(x,t)\right ]dx dt
\la{action}
\ee
where $A, B, C$ are commuting self-adjoint scalar operators and $u_1, u_2$ are $\Rr^d$ valued vector fields.
  Scalar operators act on components, $(Au(x,t))_i = Au_i(x,t)$. The operators act in $x$. In this section the operators will be based on the Laplacian via functional calculus. We view $u_1$ and $u_2$ as varying vector fields corresponding to underlying varying Lagrangian path transformations, $X_1(a,t)$ and $X_2(a,t)$ which are time-depending fluctuating diffeomorphisms of $\Rr^d$ (or $\mathbb T^d$).  The Eulerian velocities $u_1$ and $u_2$ are related to the Lagrangian transformations via 
  \be 
u_1 = \pa_t X_1\circ X_1^{-1}, \quad u_2 = \pa_tX_2\circ X_2^{-1}.
\la{u12}
\ee
We consider Eulerian variation vector fields  $\eta_1(x,t)$ and $\eta_2(x,t)$,  whch are obtained from the Lagrangian variations via  $\eta_1 = X_1'\circ X_1^{-1}$ and $\eta_2 = X_2'\circ X_2^{-1}$. Here $X_1'(a,t)$ is the variation of $X_1$, that is, a derivative of $X_1(a,t)$ with respect to a hidden parameter, at a fixed position label $a$ and fixed time $t$. Similarly $X_2'(a,t)$ is the variation of $X_2(a,t)$. These are the basic Lagrangian path fluctuation vector fields and $\eta_i$ are their Eulerian counterparts.   
Differentiating with respect to a parameter in \eqref{u12} (as in \cite{hydro})  we have by calculus
\be
u_1' = D_1\eta_1 -\eta_1\cdot\na u_1
\la{u1'}
\ee
and
\be
u_2' = D_2\eta_2 -\eta_2\cdot\na u_2.
\la{u2'}
\ee
where $D_1 = \pa_t + u_1\cdot\na $ and $D_2 = \pa_t  + u_2\cdot\na$. We thus obtain the variation of  the Eulerian velocities  due to Lagrangian fluctuations.

The action is a functional of two Lagrangian diffeomorphisms,
\be
\mathcal A = \mathcal A[X_1, X_2].
\la{mathcalalax}
\ee
We compute its variation and equate it to zero. This yields evolution equations for the velocities $u_1$ and $u_2$. We assume that
during their variation the diffeomorphisms $X_1$ and $X_2$ remain volume-preserving. This happens if (and only if) $\eta_1$ and $\eta_2$ are divergence-free (\cite{hydro}). Because $A, B, C$ are selfadjoint, we have
\be
\mathcal A' = \int_{t_1}^{t_2}\int_{\R^d}\left[(Au_1+ Cu_2)\cdot(D_1\eta_1 + \eta_1\cdot\na u_1) +
(Cu_1 + Bu_2)\cdot(D_2\eta_2 + \eta_2\cdot\na u_2)\right]dxdt.
\la{mathcalAprime}
\ee
Because $\eta_1$ and $\eta_2$ are divergence-free we have $D_1^* = -D_1$ and $D_2^* = -D_2$. Integrating by parts and equating to zero separately the contributions due to $\eta_1$ and $\eta_2$ we obtain that
\be
\pa_t w_1 + u_1\cdot\na w_1 + (\na u_1)^Tw_1 + \na p_1 = 0
\la{w1eq}
\ee
and 
\be
\pa_t w_2 + u_2\cdot\na w_2 + (\na u_2)^Tw_2 + \na p_2 = 0,
\la{w2eq}
\ee
where
\be
w_1 = Au_1 + Cu_2
\la{w1}
\ee
and
\be
w_2 = - Cu_1 - Bu_2.
\la{w2}
\ee
The pressures $p_1$, $p_2$ are due to the fact that the vanishing is only modulo gradients,  perpendicular to the space of divergence-free vector fields.

We take now $d=2$ and 
denote
\be
\sigma_1 = \na^{\perp}\cdot w_1, \quad \sigma_2 = \na^{\perp}\cdot w_2
\la{curlsw}
\ee
the curls of $w_1$ and $w_2$. It follows from \eqref{w1} and \eqref{w2} that 
\be
\left \{
\ba
\sigma_1 = A\Omega_1 + C\Omega_2\\
\sigma_2 = -C\Omega_1 - B\Omega_2
\ea
\la{sigOmega}
\right.
\ee
 holds, where
 \be
 \Omega_1 = \na^{\perp}\cdot u_1, \quad \Omega_2 = \na^{\perp}\cdot u_2
 \la{Omegas}
 \ee
 are the curls of the velocities $u_1$and $u_2$. Taking the curl of equations \eqref{w1eq} and \eqref{w2eq} we obtain that $\sigma_1$ is carried by the flow of $u_1$ and $\sigma_2$ is carried by the flow of $u_2$
  \be
 \left\{
 \ba
 \pa_t \sigma_1 + u_1\cdot\na\sigma_1 = 0,\\
\pa_t \sigma_2 + u_2\cdot\na\sigma_2 = 0.
\ea
\right.
\la{sigmasconserved}
\ee
 Now, if $AB-C^2$ is invertible, then we can invert \eqref{sigOmega}, and have
 \be
 \left\{
 \ba
 \Omega_1 = (AB-C^2)^{-1}(B\sigma_1 + C\sigma_2)\\
 \Omega_2 = (AB-C^2)^{-1}(-C\sigma_1  -  A\sigma_2).
 \ea
 \right.
 \la{Omegasig}
 \ee
 Thus the active scalars $\sigma_1$ and $\sigma_2$ determine the flow vorticities and hence the velocities.
 As a sanity check, if $C=0$ and $A= B = \mathbb I$, we obtain two independent Euler equations.
Now in the right-handed equations \eqref{active} we have, by taking the curl in \eqref{pmright} 
\be
 \left\{
 \ba
 \Omega_+ =  \Delta \left ( \mathbb K \sigma_+ + \mathbb L \sigma_-\right)\\
 \Omega_- = \Delta\left (-\mathbb L \sigma_+ - \mathbb K \sigma_-\right)
 \ea
 \right.
 \la{opmright}
 \ee
Identifying $\Omega_{+}$ with $\Omega_1$,
$\sigma_+$ with $\sigma_1$ and $\Omega_-$ with $\Omega_2$, $\sigma_-$ with $\sigma_2$ 
we see that \eqref{opmright} follows from \eqref{Omegasig} if we choose $A, B, C$ such that
\be
\left\{
\ba
(AB - C^2)^{-1}B = \Delta\mathbb K\\
(AB - C^2)^{-1}C = \Delta\mathbb L\\
A = B
\ea
\right.
\la{constr}
\ee
In order to solve for $A, B, C$ let us note that  we have, by calculus and definitions \eqref{mathbbK} and \eqref{mathbbL} that
\be
\mathbb L^2 - \mathbb K^2 = (-\Delta)^{-1}(\mathbb I -\Delta)^{-1}.
\ee
From \eqref{constr} we have first
\be
C^2-A^2 = (\mathbb I-\Delta)(-\Delta)^{-1},
\ee
and then we obtain the coefficients $A,B,C$
\be
\left\{
\ba
A = B = -\Delta \mathbb K (C^2-A^2)  = \mathbb K (\mathbb I-\Delta) = -\fr{1}{2} (-\Delta)^{-1}\\
C = -\Delta \mathbb L(C^2-A^2) = \mathbb L(\mathbb I- \Delta) = \mathbb I + \fr{1}{2}(-\Delta)^{-1}
\ea
\right .
\la{abcright}
\ee
Thus, the action for the right-handed system is
\be
\mathcal A_{right} = \int_{t_1}^{t_2}\left [(u_1(t), u_2(t))_{L^2} - \fr{1}{4}\|\Lambda^{-1}(u_1-u_2)\|_{L^2}^2\right ]dt
\la{wonkyright}
\ee 
where $\Lambda = (-\Delta)^{\fr{1}{2}}$. 
For the left-handed equations we proceed in the same manner, taking the curl in \eqref{pmleft}, and we obtain the coefficients
\be
\left\{
\ba
A= B = \mathbb L(\mathbb I -\Delta) = \mathbb I + \fr{1}{2}(-\Delta)^{-1}\\
C = -\mathbb K(\mathbb I-\Delta) = \fr{1}{2}(-\Delta)^{-1}
\ea
\right.
\la{abcleft}
\ee
and the action
\be
\mathcal A_{left} = -\int_{t_1}^{t_2} \left [\fr{1}{2} \|u_1(t)\|^2_{L^2} + \fr{1}{2}\|u_2(t)\|^2_{L^2} + \fr{1}{4}\|\Lambda^{-1}(u_1-u_2)(t)\|_{L^2}^2\right]dt.
\la{wonkyleft}
\ee
{\bf{Data Statement}} No data are associated to this work.\\
{\bf{Conflict of Interest}} The authors declare no conflict of interest.\\
{\bf{Acknowledgment}} The work of PC was partly supported by NSF grant DMS-2106528 and by a Simons Collaboration Grant 601960. The work of ZH was partly supported by a Simons Collaboration Grant 601960.

\end{document}